\documentclass[a4paper,11pt,DIV=11,
abstract=on
]{scrartcl}
\usepackage{mathtools}
\mathtoolsset{showonlyrefs}
\usepackage{amsmath, amsthm, amssymb}
\usepackage{fontenc}
\usepackage[utf8]{inputenc}
\usepackage{graphicx}
\usepackage{xargs}
\usepackage[subrefformat=parens,labelformat=parens]{subcaption}
\usepackage{algorithm, booktabs}
\usepackage[noend]{algpseudocode}
\usepackage{enumitem,listings,microtype,tikz,pgfplots}
\usepackage{hyperref}
\usepackage{environ,ifthen}
\usepackage{multirow}
\usepackage{verbatim}

\usepackage[symbol]{footmisc}

\newtheorem{theorem}{Theorem}[section]

\newtheorem{proposition}[theorem]{Proposition}
\newtheorem{lemma}[theorem]{Lemma}
\newtheorem{remark}[theorem]{Remark}
\newtheorem{example}[theorem]{Example}
\newtheorem{corollary}[theorem]{Corollary}
\setkomafont{sectioning}{\rmfamily\bfseries}
\setkomafont{title}{\rmfamily}
\setkomafont{descriptionlabel}{\rmfamily\bfseries}

\newcommand{\R}{\mathbb{R}}
\newcommand{\N}{\mathbb{N}}

\newcommand{\RR}{\mathbb{R}}
\newcommand{\Ss}{\mathbb{S}}
\newcommand{\NN}{\mathbb{N}}

\newcommand{\e}{{\,\mathrm{e}}}
\newcommand{\dist}{{\mathrm{d}}}

\newcommand{\tT}{\mathrm{T}}
\DeclareMathOperator*{\argmin}{arg\,min}
\DeclareMathOperator*{\argmax}{arg\,max}

\DeclareMathOperator{\dom}{dom}

\begin{document}
\title{On the Robust PCA and Weiszfeld's Algorithm}

\author{Sebastian Neumayer\footnotemark[1]
\and
Max Nimmer\footnotemark[1] 
\and 
Simon Setzer\footnotemark[3] 
\and 
Gabriele Steidl\footnotemark[1] \footnotemark[2]}

\maketitle
\footnotetext[1]{Department of Mathematics,
	Technische Universit\"at Kaiserslautern,
	Paul-Ehrlich-Str.~31, D-67663 Kaiserslautern, Germany,
	\{nimmer,steidl\}@mathematik.uni-kl.de.
	} 
\footnotetext[2]{Fraunhofer ITWM, Fraunhofer-Platz 1,
	D-67663 Kaiserslautern, Germany}
\footnotetext[3]{Engineers Gate, London, United Kingdom}

\begin{abstract}
    The principal component analysis (PCA) is a powerful standard tool for
		reducing the dimensionality of data.
		Unfortunately, it is sensitive to outliers so that various robust PCA variants 
		were proposed in the literature.
		This paper addresses the robust PCA  
		by successively determining the directions of lines
		having minimal Euclidean distances
		from the data points.
		The corresponding energy functional is not differentiable at a finite number of directions
		which we call anchor directions. 
		We derive a Weiszfeld-like algorithm for minimizing the energy functional
		which has several advantages over existing algorithms.
		Special attention is paid to the careful handling of the anchor directions, where
		we take the relation between local minima and 
		one-sided derivatives of Lipschitz continuous functions on submanifolds of $\mathbb R^d$ 
		into account.
		Using ideas for stabilizing the classical Weiszfeld algorithm at anchor points and the
		Kurdyka–Łojasiewicz property of the energy functional,
		we prove global convergence of the whole sequence of iterates generated by the algorithm to a
		critical point of the energy functional.
		Numerical examples demonstrate the very good performance of our algorithm.
\end{abstract}

\section{Introduction} \label{sec:intro}

Principal component analysis (PCA) \cite{Pearson1901} is an important tool for 
dimensionality reduction of data which is often 
applied as a pre-processing step, 
e.g., for classification or segmentation.
The procedure provides dimensionality reduction by projecting the data onto a linear subspace 
maximizing the variance of the projection or, 
equivalently minimizing the squared Euclidean distance error to the subspace.
More precisely, let $N \ge d $ data points $x_1,\ldots,x_N \in \mathbb R^d$ be given.
By $\| \cdot\|$ we denote the Euclidean norm and by $I_d$ the $d \times d$ identity matrix.
PCA finds a $K$-dimensional affine subspace
$\{\hat A \, t + \hat b: t \in \mathbb R^K\}$, $1 \le K \le d$,
having smallest squared Euclidean distance from the data:
\begin{equation} \label{PCA_1}
(\hat A, \hat b) \in \argmin_{A \in \mathbb R^{d,K}, b \in \mathbb R^d} \sum_{i=1}^N \min_{t \in \mathbb R^K} \|A\, t + b - x_i \|^2.
\end{equation}
While $\hat A$ and $\hat b$ in the above minimization problem are not unique,
the affine subspace itself is uniquely determined if the empirical covariance matrix only has eigenvalues of multiplicity one, and 
goes through the offset(bias) 
$\bar b := \frac{1}{N}(x_1 + \ldots + x_N)$.
Therefore we can reduce our attention to data points 
$y_i := x_i - \bar b$, $i=1,\ldots,N$ and subspaces through the origin minimizing the squared Euclidean distances to the $y_i$, $i=1,\ldots,N$.
Setting further the gradient of the inner function in \eqref{PCA_1} with respect to $t \in \mathbb R^K$ to zero 
and adding the constraint of $A$ being in the Stiefel manifold $\mathbb S_{d,K} = \{ A\in \mathbb R^{d,K}:\ A^\tT A=I_K\}$ to eliminate some redundancies
the problem reduces to
\begin{equation}
\hat A \in \argmin_{A \in \mathbb S_{d,K}} \sum_{i=1}^N \|P_A y_i \|^2,
\end{equation}
where $P_A := I_d - A A^\tT$ denotes the orthogonal projection onto 
$\mathcal{R}(A)^\perp = \mathcal{N}(A^\tT)$.

One of the most important properties of PCA is the nestedness of the PCA subspaces, i.e.,
for $K < \tilde K \le d$,
the optimal $K$-dimensional PCA subspace is contained in the $\tilde K$-dimensional one. 
In particular, the directions forming the columns of $\hat A = (\hat a_1 \; \ldots \; \hat a_K)$ 
can be found successively by computing for $k=0,1,\ldots,K-1$, 
\begin{align}
\hat a_{k+1} 
&= \argmin_{\|a\|=1} \sum_{i=1}^N \| P_a \, P_{\hat A_k} y_i \|^2   \label{pcd_succ_1}
\end{align}
or equivalently
\begin{align}
\hat a_{k+1} 
&= \argmax_{\|a\|=1} \sum_{i=1}^N  \langle a, P_{\hat A_k} y_i\rangle^2\label{pcd_succ_2}
\end{align}
where 
$\hat A_k := (\hat a_1 \; a_2 \; \ldots \; \hat a_k)$, $k=1,2,\ldots,K-1$, 
and
$P_{\hat A_0} = I_d$, see, e.g., \cite{HTF01}.
The first problem \eqref{pcd_succ_1} focuses on the minimization of the residual, while the second one \eqref{pcd_succ_2}
underlines the maximization of the variance in the PCA direction.

\begin{figure}
\centering
	\includegraphics[width=0.3\textwidth]{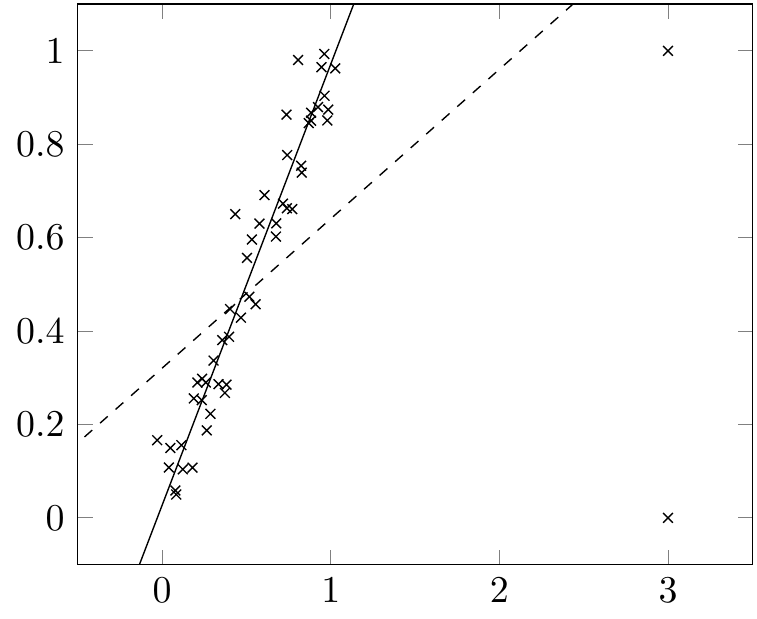}
	\caption{Illustration of the effect of outliers on the PCA. The data set consists of 
	50 points lying approximately on a line from $(0,0)^\tT$ to $(1,1)^\tT$ and two outliers. 
	The solid line is the result of PCA without outliers and the dashed line with all points.}\label{fig:outliers}
\end{figure}

Unfortunately, PCA is sensitive to outliers in the data, see  Fig.~\ref{fig:outliers}.
One possibility to circumvent the problem is to remove outliers before computing the principal components. 
However, in some contexts, outliers are difficult to identify and 
other data points are incorrectly given outlier status
forcing a large number of deletions before a reliable estimate can be found.

Therefore, quite different methods were proposed in the literature to make PCA robust,
in particular in robust statistics, see the books \cite{HR2009,MMY2006,RL1987}.
One approach consists in assigning different weights to data points based on their estimated relevance,
to get a weighted PCA, see, e.g. \cite{KKSZ08}. 
The RANSAC algorithm \cite{fischler1987random} repeatedly estimates the model parameters from a random subset of the data points until 
a satisfactory result is obtained as indicated by the number of data points within a certain error threshold. 
In a similar vein, least trimmed squares PCA models \cite{rousseeuw2005robust,podosinnikova2014robust} 
aim to exclude outliers from the squared error functional, but in a deterministic way. 
Another possible approach is to minimize the median of the squared errors as in \cite{massart1986least}.

The variational model of Candes et al. \cite{candes2011robust}
decomposes the data matrix $Y = (y_1 \; \ldots \; y_N)$ 
into a low rank and a sparse part by minimizing
$$
\argmin_{L,S} \|L\|_* + \lambda \|S\|_1 \quad \mbox{subject to} \quad Y = L + S,
$$
exploiting the nuclear norm $\|L\|_*$ of $L$ and the sum of the absolute values of entries $\|S\|_1$. 
Then $L$ can be considered as robust part, while $S$ addresses the outliers. 
Related approaches as \cite{MT2011,XCS2012} separate the low rank component from the column sparse one using similar norms. 

Another group of robust PCA approaches replaces the squared $\ell_2$ norm in PCA by the $\ell_1$ norm.
Then the minimization of the energy functional can be addressed by linear programming, see, e.g.,  
Ke and Kanade \cite{ke2005robust}. Unfortunately, this norm is not rotationally invariant. 

Mathematically interesting approaches follow \eqref{PCA_1} -  \eqref{pcd_succ_2}, but skip the squares
in the Euclidean distances and the inner products to find more robust directions.
Taking pure Euclidean distances has several consequences. First of all, the energy functionals become
non-differentiable at a finite number of subspaces spanned by matrices $A$, resp. directions $a$, which we collect within the so-called \emph{anchor set}.
Further, the offset $\hat b$ in 
\begin{equation} \label{PCA_rob_all}
(\hat A, \hat b) \in \argmin_{A \in \mathbb R^{d,K}, b \in \mathbb R^d} \sum_{i=1}^N \min_{t \in \mathbb R^K} \|A\, t + b - x_i \|.
\end{equation}
cannot be simply determined, see our small discussion in Section \ref{sec:offset}.
Let us assume that an offset $\hat b$ is given, so that we can restrict our attention to the data $y_i = x_i- \hat b$ $i=1,\ldots,N$
and \eqref{PCA_rob_all} becomes
\begin{equation} \label{pca_robust_1}
\hat A \in \argmin_{A \in \mathbb S_{d,K}} \sum_{i=1}^N  \|P_A y_i \|.
\end{equation}
Even then we lose the nested subspace property of the classical PCA, so that in particular 
\begin{align}
\hat a_{k+1} 
&= \argmin_{\|a\|=1} \sum_{i=1}^N \| P_a \, P_{\hat A_k} y_i \|   \label{pcd_robust_2}
\end{align}
and
\begin{align}
\hat a_{k+1} 
&= \argmax_{\|a\|=1} \sum_{i=1}^N  |\langle a, P_{\hat A_k} y_i\rangle|   \label{pcd_robust_3}
\end{align}
have in general nothing to do with the columns of the matrix $\hat A$ obtained in \eqref{pca_robust_1}.
Finally, the residual minimizing point of view \eqref{pcd_robust_2} leads to different results than
the variance maximizing one in \eqref{pcd_robust_3}, see Fig.~\ref{fig:ex_kwak}.

The models \eqref{pca_robust_1} - \eqref{pcd_robust_3} were considered in the literature.
The maximization of \eqref{pcd_robust_3}
was suggested with more general scalable functions than just the absolute value 
by Huber \cite[p. 203]{Huber1981} and studied in detail as PP-PCA by Li and Chen \cite{LS1985}.
It was reinvented and tackled with a greedy algorithm in \cite{kwak2008principal}
and the greedy algorithm was made more robust using median computations in \cite{HFB2014}.
For other methods in this direction, see also \cite{MT2011,nie2011robust}. 
In \cite{HFB2014} it was pointed out that the variance maximizing method in \cite{kwak2008principal} lacks a certain robustness since
it still involves mean computations. This was already demonstrated in Fig.~\ref{fig:ex_kwak}.

\begin{figure}[ht]
		\centering
		\includegraphics[width=0.5\textwidth]{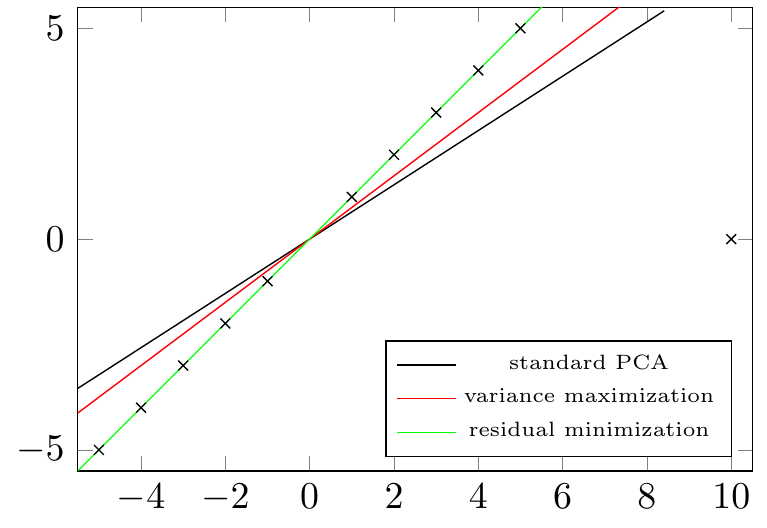}
		\caption{Results of standard PCA (black), the robust variance maximizing approach
		\eqref{pcd_robust_3} (red), and the residual minimizing method \ref{pcd_robust_2} (green). 
		The last one is closest to the line and nearly not influenced by the outlier.}\label{fig:ex_kwak}
	\end{figure}

Model \eqref{pca_robust_1} was treated by Ding et al. \cite{ding2006r}, 
where the authors circumvented the anchor set by smoothing the
original energy functional. The paper gives no convergence analysis of the proposed algorithm.
A tight convex relaxation approach for \eqref{pca_robust_1}, called REAPER was suggested in \cite{LCTZ15}. 
The relaxation replaces the condition that the symmetric positive semidefinite matrix
$A A^\tT$ has eigenvalues in $\{0,1\}$ by the condition of eigenvalues in $[0,1]$. This blows the problem size up.
Numerically the relaxed problem can be solved via an iteratively re-weighted least squares algorithm.
Usually this requires again a smoothing of the relaxed convex, but still non-differentiable functional.

In this paper, we are interested in the residual minimizing approach \eqref{pcd_robust_2}.
Recently, a minimization algorithm was  
published by Keeling and Kunisch \cite{KK2016}. 
Local convergence of their algorithm to a local minimizer 
was proved if the two parameters within the algorithm are chosen
appropriately without a concrete specification of their range.
The outcome of the algorithm is very sensitive to the choice of the parameters. 
We propose a minimization algorithm which is completely different from those in \cite{KK2016}.
It is based on ideas of the classical Weiszfeld algorithm \cite{We37} for computing the geometric median 
of points in $\mathbb R^d$ and has the advantage that
no parameters have to be tuned. In non-anchor directions the algorithm can be considered as gradient descent algorithm
on the sphere, where the length of the gradient descent is automatically given. The treatment of anchor directions relies
on one-sided directional derivatives of the energy functional. 
We show that such derivatives can be used to characterize local minima
on submanifolds of $\mathbb R^d$ of locally Lipschitz continuous functions which is interesting on its own.
We prove global convergence of our algorithm to a critical point of the energy functional, where we take special care of the anchor set.

\paragraph{Outline of the paper}
In the next Section \ref{sec:prelim}, we recall the Weiszfeld algorithm for computing the geometric median of given data points.
Properties of the energy function, critical point conditions and
the minimization algorithm are developed in Section \ref{sec:critical}.
The main part of the paper is the convergence analysis of our algorithm in Section \ref{sec:conv}.
Some remarks on the offset are given in Section \ref{sec:offset}.
Numerical examples demonstrate the performance of our algorithm in Section \ref{sec:numerics}.
The paper ends with conclusions and ideas for future work in 
Section \ref{sec:concl}.
The Appendix \ref{sec:app} provides 
a criterion for determining local minimizers of  locally Lipschitz continuous  functions
on embedded manifolds in $\mathbb R^d$ which is applied in  Section  \ref{sec:critical}.

\section{Weiszfeld's Algorithm for Geometric Median Computation} \label{sec:prelim}
We start with a small review of Weiszfeld's algorithm with two aims:
first, the geometric median usually replaces the mean as offset in robust PCA methods.
Second, having the original Weiszfeld algorithm in mind
helps to understand the basic intention of our algorithm for minimizing \eqref{pcd_robust_2}.

The \emph{geometric median} $\hat x \in \mathbb R^d$ 
of pairwise distinct points $x_i \in \mathbb R^d$, $i=1,\ldots,N$, which are not aligned,
is uniquely determined by
$$
\hat x := \argmin_x \mathcal{ E}(x) := \argmin_x \sum_{i=1}^N \|x-x_i\|.
$$
An efficient algorithm for solving the geometric median problem 
is the Weiszfeld algorithm which goes back to the Hungarian mathematician A. Vazsonyi (Weiszfeld) \cite{Vaz2002,We37}
and can be also seen as a special maximizing-minimizing algorithm, see, e.g. \cite{CIS2011}.
In \cite{Ku73,KK1962} it was recognized that the original algorithm of Weiszfeld 
fails if an iterate produced by the algorithm belongs to the so-called anchor set 
$\mathcal{A} := \{x_1,\ldots,x_N\}$ 
consisting of the points where  $\mathcal{ E}$ is not differentiable.
For bypassing the anchor  points the most natural way is to define an appropriate descent direction 
of $\mathcal{E}$ in those points \cite{OL1978,VZ2001}.
To derive the algorithm recall that
the function $\mathcal{ E}$ is convex and by Fermat's rule the vector $\hat x \in \mathbb R^d$ 
is a the minimizer of $\mathcal{ E}$ if and only if
$$
0 \in \partial \mathcal{ E} (\hat x) 
= 
\left\{
\begin{array}{ll}
\nabla \mathcal{ E} (\hat x) = \sum\limits_{i=1}^N \frac{\hat x-x_i}{\|\hat x-x_i\|}       
& \mbox{if} \;  \hat x \not \in \mathcal{A},\\[1ex]
\sum\limits_{{i=1 \atop x_i \not = \hat x}}^N \frac{\hat x-x_i}{\|\hat x-x_i\|} + \overline{B_1(0)} 
& \mbox{if} \;  \hat x \in \mathcal{A}.
\end{array}
\right. 
$$
where $\partial \mathcal{ E}$ denotes the subdifferential of $\mathcal{ E}$ and $\overline{B_1(0)}$ 
the closed Euclidean ball around zero with radius 1.
Thus, a minimizer $\hat x \not \in \mathcal{A}$ has to fulfill the fixed point equation
\begin{align}
\hat x &=  \Big(\sum_{i=1}^N \frac{1}{\|\hat x-x_i\|} \Big)^{-1} \sum_{i=1}^N \frac{x_i}{\|\hat x-x_i\|} \nonumber\\
        &= \hat x - \Big(\sum_{i=1}^N \frac{1}{\|\hat x-x_i\|} \Big)^{-1} 
				\sum_{i=1}^N \frac{\hat x-x_i}{\|\hat x-x_i\|}, \label{fix}
\end{align}
while $\hat x \in \mathcal{A}$ is a minimizer if and only if
\begin{equation} \label{global_min}
\| \sum\limits_{{i=1 \atop x_i \not = \hat x}}^N \frac{\hat x-x_i}{\|\hat x-x_i\|} \| \le 1.
\end{equation}
The Weiszfeld algorithm is an iterative algorithm which produces a sequence $\{x^{(r)}\}_r$ as follows:
if $x^{(r)} \not \in \mathcal{A}$, then we apply the Picard iteration belonging to \eqref{fix},
$$
x^{(r+1)} = x^{(r)} - \underbrace{\Big(\sum_{i=1}^N \frac{1}{\|x^{(r)}-x_i\|} \Big)^{-1} }_{s_r^{-1}}
				\underbrace{\sum_{i=1}^N \frac{x^{(r)}-x_i}{\|x^{(r)}-x_i\|}}_{\nabla \mathcal{ E} (x^{(r)})} .				
$$
This is a gradient descent step with special step size $s_r^{-1}$.
If $x^{(r)} \in \mathcal{A}$, i.e. $x^{(r)} = x_k$ for some $k\in \{1,\ldots,N\}$ 
and fulfills the minimality condition \eqref{global_min}, then the algorithm stops;
otherwise we perform a descent step in direction of the subgradient 
in $\partial {\mathcal E}(x^{(r)})$ which is closest to zero
$$
x^{(r+1)} := x^{(r)} - \Big(\sum\limits_{i=1 \atop i \not = k}^N \frac{1}{\|\hat x-x_i\|} \Big)^{-1} \left(1- \frac{1}{\|G_k\|} \right) G_k
$$
where $G_k := \sum\limits_{{i=1 \atop i \not = k}}^N \frac{x_k-x_i}{\| x_k-x_i\|} \in \partial {\mathcal E}(x_k)$.

Local and asymptotic convergence rates of the Weiszfeld algorithm were given in \cite{Ka74} 
and a non-asymptotic sublinear convergence rate was proved in \cite{BS2015}.
The very good performance of Weiszfeld's algorithm in comparison with the parallel proximal point algorithm
was shown in \cite{SST2012} and a projected Weiszfeld algorithm was established in \cite{NNS2017}.
Keeling and Kunisch \cite{KK2016} suggested another stable algorithm for finding the geometric mean
based on criticizing the behavior of the original Weiszfeld algorithm in anchor points and not taking
its stabilized versions into account.
A good reference on past and ongoing research in this direction is \cite{BS2015} and the references therein.

\section{Weiszfeld-like Algorithm for Robust PCA} \label{sec:critical}
We consider the minimization approach \eqref{pcd_robust_2}.
First of all we see in the next remark that the direction $a_{k+1}$ is indeed
perpendicular to the previous directions $\{a_1,\ldots,a_k\}$.

\begin{remark} \label{lem:start}
Let $\varphi:\mathbb R_{\ge 0} \rightarrow \mathbb R$ be a strictly increasing function.
In our application we are interested in $\varphi(x) = x^\frac12$.
For any $z_i \in \mathbb R^d$, $i=1,\ldots,N$ its holds 
\begin{equation} \label{noch}
\argmin_{\|a\|=1 } \sum_{i=1}^N \varphi\big(\| P_a z_i \|^2\big)  \; \in \; \mathrm{span}\{z_i:i=1,\ldots,N\}
\end{equation}
by the following reasons:
Every $a \in \mathbb{R}^d$ with $\|a\|=1$ can be written as 
		\[
			a = \frac{\tilde a+ \tilde a_\perp}{\|\tilde a+\tilde a_\perp\|}, 
		\]
where $\tilde a \in \mathrm{span}\{z_i:i=1,\ldots,N\}$ and $\tilde a_\perp$ is in the orthogonal complement of $\mathrm{span}\{z_i:i=1,\ldots,N\}$.
Then we have for every $z \in \mathrm{span}\{z_i:i=1,\ldots,N\}$,
		\[
			\|P_{a} z\|^2 
			= \|z\|^2 - \langle a, y\rangle^2 
			= \|z\|^2 -\frac{\langle \tilde a, z\rangle^2 }{\|\tilde a\|^2 + \|\tilde a_\perp\|^2} 
			\ge 
			\|z\|^2 - \frac{\langle \tilde a, z\rangle^2 }{\|\tilde a\|^2 }
		\]
		with equality if $\|\tilde a_\perp\| = 0$.
Since $\varphi$ is strictly increasing, any minimizer $\hat a$ must be in $\mathrm{span}\{z_i:i=1,\ldots,N\}$.
\hfill $\Box$
\end{remark}

We have to deal with the function
\begin{equation} \label{E}
E(a) := \sum_{i=1}^N E_i(a) =  \sum_{i=1}^N \| P_a y_i \| .
\end{equation}
This function is continuously differentiable on $\mathbb R^d$ except for $a \in \mathbb R^d$
satisfying $\| P_a y_k \| = \|(I_d - a a^\tT) y_k\| = 0$ for some $k \in \{1,\ldots,N\}$.
This is equivalent to
$y_k = a \langle a, y_k\rangle$
and for
$a \in  \mathbb S^{d-1}$ to $a \in \{ \pm \frac{y_k}{\|y_k\|} \}$. Let
$$
\mathcal{A} := \{\pm \frac{y_i}{\|y_i\|}: i=1,\ldots,N\}
$$
denote this set of directions on $\mathbb S^{d-1}$, where $E$ is not differentiable.
Similarly as in Weiszfeld's algorithm, we call it \emph{anchor set}.

The following theorem collects important properties of $E$.
The third property relies on the relation between one-sided derivatives
and local minima of Lipschitz continuous functions 
on embedded manifolds in $\mathbb R^d$. The definition of one-sided derivatives
and a theorem local minima can be characterized by is given in Appendix \ref{sec:app}. In our case the embedded manifold is the sphere $\mathbb{S}^{d-1}\coloneqq \{ a\in\mathbb{R}^d:\ \|a\|=1 \}$.

\begin{theorem} \label{propE}
Let $E$ defined by \eqref{E}.
\begin{enumerate}
\item 
The function $E$ is locally Lipschitz continuous on $\mathbb R^{d}$.
\item
For $a \in \mathbb S^{d-1} \backslash	\mathcal{A}$, it holds
\begin{equation} \label{def_C}
\nabla E(a) = - P_a \, C_a \, a,  
\qquad 
C_{a} := \sum_{i=1}^N \frac{1}{\|P_a y_i\|} y_i y_i^\tT,
\end{equation}
and $\nabla E(a)$ is in the tangent space $T_a \mathbb S^{d-1}$ of $\mathbb S^{d-1}$ at $a$. 
\item 
A direction $a \in \mathcal{A}$ is a local minimizer of $E$ if
\begin{equation} \label{condition}
 \|G_{a,\mathcal{K}}\|  <  \sum_{k \in \mathcal{K}} \|y_k\|,  
\end{equation}
where
$\mathcal{K} := \{ k \in \{1,\ldots,N\}: \|P_a y_k\| = 0 \}$ and
$$
G_{a,\mathcal{K}} := P_a C_{a,\mathcal{K}} a , 
\qquad 
C_{a,\mathcal{K}} := \sum_{i \not \in \mathcal{K} } \frac{1}{\|P_a y_i\|} y_i y_i^\tT.
$$
\end{enumerate}
\end{theorem}

\begin{proof}
1. It suffices to show the property for the summands $E_i$. For an arbitrary fixed $a \in \mathbb R^{d}$,
let $\|a- a_i\| \le \varepsilon$, $i=1,2$.
Then we obtain
\begin{align*}
|E_i(a_1) - E_i(a_2)| &= | \, \|P_{a_1} y_i\| - \|P_{a_2} y_i\| \, |
\le \|P_{a_1} y_i - P_{a_2} y_i\| \\
&\le \|a_1 a_1^\tT - a_2 a_2^\tT\|_F \|y_i\| \\
&= \frac12 \|(a_1 - a_2)(a_1^\tT + a_2^\tT)  + (a_1+ a_2)(a_1^\tT - a_2^\tT) \|_F \| y_i\|\\
&\le 2 (\|a\| + \varepsilon) \|y_i\| \|(a_1 - a_2)\|_F. 
\end{align*}
2. By straightforward computation we obtain at points  $a \in \mathbb R^{d}$, where $E$ is differentiable,
 \begin{align*}
\nabla E_i(a) &= - \frac{1}{\|P_a y_i\|} \left( P_a y_i y_i^\tT a + y_i y_i^\tT P_a a \right).
\end{align*}
The second summand vanishes for $a \in \mathbb S^{d-1}$ which yields \eqref{def_C}. 
Since $P_a$ projects to the space orthogonal to $a$
the gradient $\nabla E(a)$ lies in $T_a \mathbb S^{d-1}$.
\\[1ex]
3. For $a \in \mathcal{A}$ we have $a \in \pm \frac{y_k}{\| y_k\|}$ for $k \in \mathcal{K}$.
Then the one-sided directional derivative of $E_k$ at $a \in \mathcal{A}$ 
in direction $h \in T_a S^{d-1}$ reads as
\begin{align*}
D E_k (a;h) 
&= \lim_{\alpha\downarrow 0} \frac{E_k(a+\alpha a) - E_k(a)}{\alpha}
= \lim_{\alpha\downarrow 0} \frac{\|(I-(a+\alpha h)(a+\alpha h)^\tT) y_k\|}{\alpha} \\
&= \lim_{\alpha\downarrow 0} \frac{\|\alpha a h^\tT y_k + \alpha h a^\tT y_k + \alpha^2 h h^\tT y_k\|}{\alpha}\\
&= \| (a h^\tT  + h a^\tT) y_k \| = \| h a^\tT y_k \| = \|h\| \|y_k\|.
\end{align*}
For $i \not \in \mathcal{K}$ we have by part 2 of the proof
\begin{align*}
D E_i (a;h) 
&= 
\langle \nabla E_i(a),h \rangle 
= 
- \left\langle  P_a  \frac{1}{\|P_a y_i\|} y_i y_i^\tT a , h \right \rangle
\end{align*}
so that in summary
\begin{align*}
D E (a;h) = \sum_{k \in \mathcal{K}} \| y_k\| \|  h  \| - \langle P_a C_{a,\mathcal{K}} a,h \rangle.
\end{align*}
Since $E$ is locally Lipschitz continuous on $\mathbb R^d$, we conclude by Theorem \ref{prop:loc_min} 
that $a \in \mathcal{A}$ is a local minimizer if
\begin{equation} \label{min_1}
\langle P_a C_{a,\mathcal{K}} a,h \rangle
< 
\sum_{k \in \mathcal{K}}  \| y_k\| \|  h  \| 
\end{equation}
for all $h \in T_a \mathbb S^{d-1}$. Since $P_a C_{a,\mathcal{K}} a \in T_a \mathbb S^{d-1}$ this equivalent to
$$
 \| P_a C_{a,\mathcal{K}} a \| < \sum_{k \in \mathcal{K}}  \| y_k\|.
$$ 
\end{proof}

To establish a Weiszfeld-like algorithm, we consider again two cases:

If $a \not \in \mathcal{A}$, then
$
0 = \nabla E(a) = -P_a C_a a
$
can be rewritten as the fixed point equation
\begin{align}
a &= (a^\tT C_a a)^{-1} C_a a \\
&= a + (\underbrace{a^\tT C_a a}_{s_a} )^{-1} P_a C_a a.
\end{align}
This gives rise to the gradient descent step on $\mathbb S^{d-1}$:
\begin{align}
a^{(r+1)} = \frac{C_{a^{(r)}} a^{(r)}}{\|C_{a^{(r)}}a^{(r)}\|},
\end{align}
where the factor $s_{a^{(r)}}$ cancels out when projecting on $\mathbb S^{d-1}$.
This also appears in the algorithm proposed by Ding et al. \cite{ding2006r}
from another point of view.

If $a \in \mathcal{A}$ and $\|G_{a,\mathcal{K}}\|  >  \sum_{k \in \mathcal{K}} \|y_k\|$, then 
we suggest to use
$$
G_{a,\mathcal{K}} \left(1- \frac{ \sum_{k\in \mathcal{K}} \|y_k\|}{\| G_{a,\mathcal{K}} \|} \right)
$$
instead of the gradient as descent direction which results in the iteration
$$
a^{(r+\frac12)} := a^{(r)} + s_{a^{(r)},\mathcal{K}}^{-1} \left(1- \frac{\sum_{k\in \mathcal{K}} \|y_k\|}{\|G_{a^{(r),\mathcal{K}} }\|} \right)
$$ 
with
$$
s_{a,\mathcal{K}} := a^\tT C_{a,\mathcal{K}} a =  \sum_{i \not \in \mathcal{K} } \frac{\langle a, y_i \rangle^2}{\|P_a y_i\|}
$$
and subsequent orthogonal projection onto $\mathbb S^{d-1}$.
In summary, we obtain Algorithm \ref{alg:onedir}.
%
\begin{algorithm}[htb]
	\caption{Algorithm for Minimizing $E$ over $\mathbb S^{d-1}$}\label{alg:onedir}
	\begin{algorithmic}
	\State \textbf{Input:}
		$y_i \in \mathbb R^d$,  $i=1,\ldots,N$ pairwise distinct with positive definite covariance matrix
	\State \hspace{1.2cm} $a^{(0)} \in \mathbb R^d$
		
	\State r = 0
		\Repeat 
		\If{$\|P_{a^{(r)}} y_k\| \not = 0$ for all $k\in \{1,\ldots,N\}$}
		\vspace{0.1cm}
		
		\State
		$
		C_{a^{(r)} } := \sum_{i=1}^N \frac{1}{\| P_{a^{(r)}} y_i\|} y_i y_i^\tT
		$
		\vspace{0.1cm}
		\State
		$
		a^{(r+1)} := \frac{C_{a^{(r)}} a^{(r)}}{\|C_{a^{(r)}} a^{(r)}\|}	
		$
		\\
		\ElsIf{$\|P_{a^{(r)}} y_k\|  = 0$ for $k\in {\mathcal K} \subset \{ 1,\ldots,N \}$}
		\vspace{0.1cm}
		
		\State
		$
		C_{a^{(r)},\mathcal{K}} :=  \sum_{i \not \in \mathcal{K} } \frac{1}{\|P_{a^{(r)}} y_i\|} y_i y_i^\tT
		$
		\vspace{0.1cm}
		\State
		$G_{a^{(r)},\mathcal{K}} := P_{a^{(r)}} C_{a^{(r)},\mathcal{K}} a^{(r)} $
		\\
		
		\If{$\|G_{a^{(r)},\mathcal{K}}\| \le \sum\limits_{k \in \mathcal{K}} \|y_k\|$}
		\State {termination}
		
		\Else
		\State
		$s_{a^{(r)},\mathcal{K}} :=  \sum\limits_{i \not \in \mathcal{K} } \frac{\langle a^{(r)}, y_i \rangle^2}{\|P_{a^{(r)}} y_i\|}$
		\State	
		$a^{(r+\frac12)} := a^{(r)} + s_{a^{(r)},\mathcal{K}}^{-1} \left(1- \frac{\sum\limits_{k\in \mathcal{K}} \|y_k\|}{\|G_{a^{(r)},\mathcal{K} }\|} \right) 
		G_{a^{(r)},\mathcal{K}}$
		\State
		$a^{(r+1)} = \frac{ a^{(r+ \frac12)} }{\| a^{(r+ \frac12)}\|}$
		\EndIf
	\EndIf	
	\State $r\rightarrow r+1$		
	\Until a stopping criterion is reached	
	\end{algorithmic}
\end{algorithm}

\section{Convergence Analysis} \label{sec:conv}
In this section, we show that that sequence generated by the Algorithm \ref{alg:onedir}
converges to a critical point of $E$,
where we say that $a \in \mathbb S^{d-1}$ 
is a \emph{critical point} of $E$ on $\mathbb S^{d-1}$
if one of the following conditions is fulfilled:
\begin{itemize}
 \item[i)]  
$a \not \in \mathcal{A}$ and $-\nabla E(a) = P_a C_a a = 0$ .
\item[ii)]
$a \not \in \mathcal{A}$ and
$\| G_{a,\mathcal{K}} \| \le  \sum_{k\in \mathcal{K}} \|y_k\|$.
\end{itemize}
We need four lemmata and apply a theorem of Attouch, Bolte and Svaiter \cite{ABSF13}
on the convergence of functions having the Kurdyka–Łojasiewicz property.

\begin{lemma}\label{alg_stop}
For the sequence $\{a^{(r)}\}_r$  produced by Algorithm \ref{alg:onedir}
we have $a^{(r+1)} = a^{(r)}$ if and only if $a^{(r)}$ is a critical point of $E$ on $\mathbb S^{d-1}$.
If the iteration stops after finitely many steps, then it has reached a critical point.
\end{lemma}

\begin{proof}
1. Let $a^{(r+1)} =  a^{(r)} = a$. 
If $a$ is not in the anchor set, this implies 
$\frac{C_a a}{\|C_a a\|} = a$ and hence $P_a C_a a = \|C_a a\| P_a a = 0$.
If $a$ is in the anchor set, then relation in ii) must be fulfilled by the stopping condition.
\\
2. 
Let $a^{(r)} \in \mathbb S^{d-1}$ be a critical point of $E$. 
If $a^{(r)}$ is not in the anchor set, then by definition
$0 = P_{a^{(r)}} C_{a^{(r)}} a^{(r)} = C_{a^{(r)}} a^{(r)} - a^{(r)} \, (a^{(r)})^\tT C_{a^{(r)}} a^{(r)}$ so that
$$ 
a^{(r+1)} =  \frac{C_{a^{(r)}} a^{(r)}}{\Vert C_{a^{(r)}} a^{(r)} \Vert} 
=  
\frac{(a^{(r)})^\tT C_{ a^{(r)} } a^{(r)}) \, a^{(r)}}{\|\left( (a^{(r)})^\tT C_{a^{(r)}} a^{(r)}\right) \, a^{(r)}\|} = a^{(r)}.
$$
If $a^{(r)}$ is in the anchor set, then $\| G_{{a^{(r)}},\mathcal{K}} \| \le  \sum_{k\in \mathcal{K}} \|y_k\|$
and the iteration stops by definition, i.e. $a^{(r+1)} = a^{(r)}$.
\end{proof}

\begin{lemma}\label{abstieg}
Let $\{a^{(r)}\}_{r}$ be the sequence generated by Algorithm \ref{alg:onedir}. 
If $a^{(r+1)} \not = a^{(r)}$, then $E(a^{(r+1)}) < E(a^{(r)})$.
The sequence $\{ E(a^{(r)}) \}_r$ converges to some value $\hat E \ge 0$.
\end{lemma}

\begin{proof}
If the sequence of function values decreases, its convergence follows immediately from the fact that E is bounded from below by zero.
To show the decrease property, we set 
$a := a^{(r)}$, $\bar a := a^{(r+\frac12)}$ and $\tilde a = a^{(r+1)}$ 
and abbreviate
$$
G:= G_{ a^{(r)},\mathcal{K} }, \quad C := C_{a^{(r)},\mathcal{K}} \quad \mbox{and} \quad s:= s_{a^{(r)},\mathcal{K}},
$$
where $\mathcal{K}$ is the empty set if $a^{(r)}$ is not an anchor direction.

\textbf{Case 1}: Let $a\notin \mathcal{A}$ be a non-anchor direction.
For $u\geq 0, v>0$ it holds $u-v\leq \frac{u^2-v^2}{2v}$ so that
	\begin{align}
	E(\tilde a) - E(a) &=
	\sum_{i=1}^N\left(  \| P_{\tilde a} y_i\| - \|P_{a}y_i\| \right)\\
	&\leq 
	\sum_{i=1}^N \frac{\|P_{\tilde a} y_i\|^2 - \|P_{a} y_i\|^2}{2\|P_{a} y_i\|}
	= 
	\sum_{i=1}^N  \frac{\|\tilde a \tilde a^\tT y_i - y_i\|^2 - \|a a^\tT y_i - y_i\|^2}{2\|P_{a}y_i\|} .
	\end{align}
Using $\|u-v\|^2 - \|w-v\|^2 = 2\langle u-w,u-v \rangle - \|u-w\|^2$ we get
	\begin{align}
	E(\tilde a) - E(a)
	&\leq 
	\sum_{i=1}^N  \frac{1}{ \| P_{a} y_i\|} \langle \tilde a \tilde a^\tT y_i - a a^\tT y_i, \tilde a \tilde a^\tT y_i -y_i \rangle
		- \sum_{i=1}^N  \frac{\|\tilde a \tilde a^\tT y_i - a a ^\tT y_i\|^2}{2\|P_{a}y_i\|}\\
	&= 
	\sum_{i=1}^N  \frac{\langle a a ^\tT y_i, P_{\tilde a}y_i  \rangle}{\|P_{a}y_i\|} 
	-  \sum_{i=1}^N  \frac{\|\tilde a \tilde a^\tT y_i - a a ^\tT y_i\|^2}{2\|P_{a}y_i\|}
	\label{eq:wf_proof_mid}\\
	&= a^\tT P_{\tilde a} C a  -\sum_{i=1}^N  \frac{\|\tilde a \tilde a^\tT y_i - a a ^\tT y_i\|^2}{2\|P_{a}y_i\|}\\
	&= \|C a\| a^\tT P_{\tilde a} \tilde a
	-\sum_{i=1}^N  \frac{\|\tilde a \tilde a^\tT y_i - a a ^\tT y_i\|^2}{2\|P_{a}y_i\|}, \label{eq:wf_descent_badcase_1}
	\end{align}
	which finally implies
	\[
	E(\tilde a) - E(a)  \le -\sum_{i=1}^N  \frac{\|\tilde a \tilde a^\tT y_i - a a ^\tT y_i\|^2}{2\|P_{a}y_i\|}.
	\]
	Since $a,\tilde a \in \mathrm{span}(Y)$ the right-hand side is strictly negative
	except for $\tilde a = \pm a$ which was excluded.

\textbf{Case 2}: Let $a\in\mathcal{A}$, i.e., $\|P_a y_k\| = 0$ for $k \in \mathcal{K} \not = \emptyset$
and
$$
\| G \| > \sum_{k \in \mathcal{K}} \|y_k\| =: \alpha.
$$
From $P_a y_k = 0$, $k \in \mathcal{K}$, 
i.e., $y_k = a (a^\tT y_k)$ we obtain $\|y_k\| = |a^\tT y_k|$.
Since $\bar a = a + S^{-1} \left(1- \frac{\alpha}{\|G\|}\right) G$ and $a \perp G$ we have
\begin{align} 
\| \bar a\|^2  &= 1 + s^{-2}\left( 1-\frac{\alpha}{\|G\|} \right)^2 \|G\|^2 > 1\\
1- \frac{1}{\|\bar a\|^2} &= \frac{ (\|G\| - \alpha)^2}{s^2 \|\bar a\|^2} =: \mu^2.      \label{need2}
\end{align}
We have to estimate
\begin{equation}
E(\tilde a) - E(a)
= 
\sum_{i\not \in \mathcal{K}} \left( \| P_{\tilde a} y_i\| - \|P_{a} y_i\| \right) + \sum_{k  \in \mathcal{K}} \|P_{\tilde a} y_k\|  .
\end{equation}
First, we get for $k \in {\mathcal K}$,
\begin{equation}
	\| P_{\tilde a} y_k\|^2
	=
	y_k^\tT \left( I-\frac{\bar a \bar a^\tT }{\|\bar a\|^2} \right) y_k
	= 
	\|y_k\|^2 - \frac{y_k^\tT a a^\tT y_k}{\|\bar a\|^2} 
	= 
	\mu^2 \|y_k\|^2.
\end{equation}
so that
\begin{equation} \label{need1}
 \sum_{k  \in \mathcal{K}} \|P_{\tilde a} y_k\| = \mu \, \alpha.
\end{equation}
Replacing the sum over $\{1,\ldots,N\}$ in the first step of the proof 
by those over $\{1,\ldots,N\} \backslash  \mathcal{K}$ we get instead of \eqref{eq:wf_descent_badcase_1}
	\begin{align}
	\sum_{i\not \in \mathcal{K}}\left(  \| P_{\tilde a} y_i\| - \|P_{a}y_i\| \right)
	&\leq 
	 a^\tT \left( I-\frac{\bar a \bar a^\tT}{\|\bar a\|^2}  \right) C a 
	-\sum_{i\not \in \mathcal{K}} \frac{\|\tilde a \tilde a^\tT y_i - a a ^\tT y_i\|^2}{2\|P_{a}y_i\|}. \label{eq:wf_descent_badcase}
	\end{align}
	By definition of $\bar a$ and \eqref{need2} we can rewrite
	\begin{align}
	a^\tT \left( I- \frac{\bar a \bar a^\tT}{\|\bar a\|^2} \right) C a
	&=  a^\tT C a 
	- \frac{1}{\|\bar a\|^2} \left( a^\tT C a + s^{-1}\left( 1-\frac{\alpha}{\|G\|} \right) G^\tT C a \right)\\
	&= s - \frac{1}{\|\bar a\|^2} \left( s + s^{-1}\left( 1-\frac{\alpha}{\|G\|} \right)\|G\|^2 \right)\\
	&= \mu^2 s - \frac{1}{\|\bar a\|^2}s^{-1}\left( 1-\frac{\alpha}{\|G\|} \right)\|G\|^2\\
	&= \mu^2 s \left( 1 - \frac{\|G\|}{\|G\| - \alpha} \right)
	=  - \mu^2 s \, \frac{\alpha}{\|G\| -\alpha}\\
	&= -\mu \frac{\alpha}{\|\bar a\|}. \label{need3}
	\end{align}
	For the second sum in \eqref{eq:wf_descent_badcase} we get
	\begin{align}
	-\sum_{i\not \in {\mathcal K}} \frac{\|\tilde a \tilde a^\tT y_i - a a^\tT y_i\|^2}{2\|P_{a}y_i\|}
	= -\frac 12 \left( \tilde a^\tT C \tilde a - 2 \tilde a^\tT a a^\tT C\tilde a + a^\tT C a \right).
	\end{align}
	Application of $a^\tT CG = \|G\|^2$ and of the definition of $\tilde a$ leads to
	\begin{align*}
		\tilde a^\tT C \tilde a 
		& = \frac{1}{\|\bar a\|^2} \big( s + 2 s^{-1}(\|G\|-\alpha)\|G\| +
		s^{-2}(\|G\| - \alpha)^2\frac{1}{\|G\|^2}G^\tT C G \big)
	\end{align*}
	and
	\[
	\tilde a^\tT a a ^\tT C \tilde a = \frac{1}{\|\bar a \|^2}\left( s + s^{-1}(\|G\| - \alpha)\|G\| \right).
	\]
	Hence we obtain
	\begin{align}
			-\sum_{i \not \in \mathcal{K}}^N \frac{\|\tilde a \tilde a^\tT y_i - a  a^\tT y_i\|^2}{2\|P_{a}y_i\|}
			&=
			-\frac 12 \frac{1}{\|\bar a\|^2} \left( (\|\bar a\|^2 -1) s + s^{-2} (\|G\| - \alpha)^2\frac{1}{\|G\|^2} G^\tT C G \right)\\
			&= -\frac 12 \mu^2 \left(s + \frac{1}{\|G\|^2} G^\tT C G \right). 			
	\end{align}
	Since $C$ symmetric positive definite we conclude by Young's inequality
	\begin{align*}
		s + \frac{1}{\|G\|^2}G^\tT C G
		= \|C^{\frac12}a\|^2 + \|\frac{1}{\|G\|} C^{\frac 12}G\|^2
		\geq  2\frac{1}{\|G\|} a^\tT C^{\frac12}C^{\frac12}G
		=2\|G\|,
	\end{align*}
	so that
	$$
	-\sum_{i \not \in \mathcal{K}} \frac{\|\tilde a \tilde a^\tT y_i - a  a^\tT y_i\|^2}{2\|P_{a}y_i\|}
	\le 
	- \mu^2 \|G\|.
	$$
	Combining this equation with \eqref{need1}, \eqref{eq:wf_descent_badcase} and  \eqref{need3},
	and using that $\| \bar a\| > 1$, we obtain
	\begin{align*}
		E(\tilde a) - E(a)
		&\leq 
		-\mu \frac{\alpha}{\|\bar a\|} - \mu^2\|G\|  + \mu \alpha = \mu \left( \alpha (1- \frac{1}{\|\bar a\|}) - \mu \|G\| \right) \\
		& < \mu^2( \alpha - \|G\|) < 0.
	\end{align*}
\end{proof}	

\begin{lemma}\label{lem:conv_1}
Let $\{a^{(r)}\}_{r}$ be an infinite sequence generated by Algorithm \ref{alg:onedir}.
Then we have 
\begin{equation} \label{small}
\lim_{r\rightarrow \infty} \| a^{(r+1)} - a^{(r)}\| = 0.
\end{equation}
The set of accumulation points is compact and connected.
\end{lemma}

\begin{proof} 
Since the number of anchor directions is finite, we can choose $R$ large enough such that all iterates
$a^{(r)}$, $r \ge R$ are no anchor directions.
Since the projection $\Pi_{\Ss^{d-1}}$ onto the unit sphere is non-expansive 
for points not in the interior of the unit ball, we obtain
\begin{align*}
	\|a^{(r+1)} - a^{(r)}\| &= \left\| \Pi_{\Ss^{d-1}}\left(a^{(r)} + s_{a^{(r)}}^{-1}P_{a^{(r)}}C_{a^{(r)}}a^{(r)}\right) - \Pi_{\Ss^{d-1}}\left(a^{(r)}\right)\right\|\\
	&\leq 
	\frac{\|P_{a^{(r)}}C_{a^{(r)}} a^{(r)}\|}{s_{a^{(r)}}}.
\end{align*}
We show that all  accumulation points of $\{\beta_r\}_r$ with $\beta_r := \|a^{(r+1)} - a^{(r)}\|$ are zero.
Note that such accumulation points exist, since $\mathbb S^{d-1}$ is compact so that the sequence is bounded from below and above.
Let $\{\beta_{r_j}\}_j$ converge to $\hat \beta$ which is then also true for every subsequence.
Let $\{\beta_{r_{j_i}}\}_i$ by any subsequence for which $\{a^{(r_{j_i})} \}_i$ converges to an accumulation point $\hat a$.
For simplicity of notation, we skip the second index $i$.
We distinguish two cases:
\begin{enumerate}
	\item Let $\hat a \notin \mathcal{A}$ be a non-anchor direction.
	Then the update operator 
	$T(a) = \frac{C_a a}{\Vert C_a a\Vert}$ of the algorithm
	is continuous in $\hat a$ so that
	$\lim_{j\to\infty}a^{(r_j+1)} = \lim_{j\to\infty} T(a^{(r_j)}) = T(\hat a)$.
	By Lemma~\ref{abstieg} and continuity of $E$, we get
	\begin{align*}
		\hat E &= \lim_{j\to\infty} E(a^{(r_j)})  = E(\hat a)\\
		\hat E &= \lim_{j\to\infty} E(a^{(r_j+1)})  = E\left( T(\hat a) \right)
	\end{align*}
	so that $\hat a= T(\hat a)$. 
	This in turn yields 
	$P_{\hat a}C_{\hat a} \hat a= \|C_{\hat a} \hat a\|P_{\hat a} \hat a =0$. 
	Since the $\|P_{a^{(r_j)}}y_i\|$ are bounded from above,
	and since $a^{(r)}\in \mathrm{span}(Y)$ for all $r$ we conclude
	that $s_{a^{(r_j)}}$ is bounded from below.
	Taking the continuity of the involved operators in $\hat a$ into account, this implies
	\[\lim_{j \to\infty} \frac{\|P_{a^{(r_j)}}C_{a^{(r_j)}} a^{(r_j)}\|}{s_{a^{(r_j)}}} =0.\]
	
	\item Let $\hat a \in\mathcal{A}$ be an anchor direction. Then it holds
	\[
		\lim_{j\to\infty} s_{a^{(r_j)}} = \lim_{j\to\infty} \sum_{i=1}^N \frac{(y_i^\tT a^{(r_j)})^2}{\|P_{a^{(r_j)}}y_i\|} = \infty,
	\]
	while
	\[
		\|P_{a^{(r_j)}}C_{a^{(r_j)}} a^{(r_j)}\| = \left\|\sum_{i=1}^N y_i^\tT a^{(r_j)} \frac{P_{a^{(r_j)}}y_i}{\|P_{a^{(r_j)}}y_i\|} \right\| 
		\leq \sum_{i=1}^N |y_i^\tT a^{(r_j)}| \leq \sum_{i=1}^N \|y_i\|,
	\]
	so that
	\[
		\lim_{j\to\infty} \frac{\|P_{a^{(r_j)}}C_{a^{(r_j)}} a^{(r_j)}\|}{s_{a^{(r_j)}}} = 0.
	\]
\end{enumerate}
This proves \eqref{small}.
By Ostrowski's Theorem,
the set of accumulation points of the sequence of iterates is compact and connected.
\end{proof}

\begin{lemma}\label{lem:conv_2}
Let $\hat a$ be an anchor direction.
Let $T$ denote the iteration function of Algorithm \ref{alg:onedir}.
Then
\[
\lim_{a \rightarrow \hat a} \frac{\|T (a) - \hat a\|}{\|a-\hat a\|} 
= 
\frac{\|G_{\hat a,\mathcal{K} }\|}{\sum_{k \in \mathcal{K}} \|y_k\|}.
\]
\end{lemma}

\begin{proof} 
For simplicity of notation, we assume that $\mathcal{K} = \{k\}$ and
without loss of generality $\hat a = y_k/\|y_k\|$. 
We set
 \[
 	T(a) = \frac{C_aa}{\|C_aa\|} = \frac{a + \frac{1}{s_a}P_aC_aa}{\|a + \frac{1}{s_a}P_aC_aa\|} \eqqcolon \frac{T_a}{\|T_a\|}.
 \]
 Similarly as in the proof of Lemma \ref{lem:conv_1}, Case 1, 
we have that $P_aC_aa$ is bounded from above and $\lim_{a\to \hat a} s_a = \infty$ so that $\lim_{a\to \hat a} \|T_a\|= 1$. 
We calculate
 \begin{align}
 	\frac{\|T(a) - \hat a\|^2}{\|a-\hat a\|^2}
 	&= \frac{\|a - \|T_a\|\hat a + \frac{1}{s_a}P_aC_aa\|^2}{\|a-\hat a\|^2\|T_a\|^2}\\
 	&= \frac{\|a - \|T_a\|\hat a\|^2}{\|T_a\|^2\|a-\hat a\|^2} + \frac{2\langle a-\|T_a\|\hat a,\frac{1}{s_a}P_aC_aa \rangle 
	+ \frac{1}{s_a^2}\|P_aC_aa\|^2}{\|T_a\|^2\|a-\hat a\|^2}.\label{eq:proof_terms}
 \end{align}
 The first term can be rearranged as
 \begin{align}
 	\frac{\|a - \|T_a\|\hat a\|^2}{\|T_a\|^2\|a-\hat a\|^2}
 	&= \frac{1}{\|T_a\|^2} \frac{\|a-\hat a + (1-\|T_a\|)\hat a\|^2}{\|a-\hat a\|^2}\\
 	&= \frac{1}{\|T_a\|^2} \frac{\|a-\hat a\|^2 + 2\langle a-\hat a, (1-\|T_a\|)\hat a) \rangle + (1-\|T_a\|)^2\|\hat a\|^2}{\|a-\hat a\|^2}\\
 	&= \frac{1}{\|T_a\|^2} \left( 1 + \frac{2\langle a, (1-\|T_a\|)\hat a) \rangle - 2(1-\|T_a\|) + (1-\|T_a\|)^2}{\|a-\hat a\|^2} \right)\\
 	&= \frac{1}{\|T_a\|^2} \left( 1 + \frac{2\langle a, (1-\|T_a\|)\hat a) \rangle - 1 + \|T_a\|^2}{\|a-\hat a\|^2} \right)\\
 	&= \frac{1}{\|T_a\|^2} \left( 1 + \frac{2\langle a, (1-\|T_a\|)\hat a) \rangle + \frac{1}{s_a^2}\|P_aC_aa\|^2}{\|a-\hat a\|^2} \right). \label{weiter}
 \end{align}
 By Taylor approximation of $\sqrt{1+x}$ at $x=0$ we get
 \[
 	1 - \|T_a\| = 1 - \sqrt{1 + \frac{1}{s_a^2}\|P_aC_aa\|^2} = - \frac{1}{2s_a^2}\|P_aC_aa\|^2 + \mathcal{O}(\frac{1}{s_a^4}\|P_aC_aa\|^4)
 \]
 Plugging this into \eqref{weiter} yields
 \begin{align}
 	\frac{\|a - \|T_a\|\hat a\|^2}{\|T_a\|^2\|a-\hat a\|^2}
 	= \frac{1}{\|T_a\|^2} \left( 1 + \frac{(1-\langle a,\hat a \rangle)\|P_aC_aa\|^2 
	+ 2 \langle a,\hat a\rangle \mathcal{O}(\frac{1}{s_a^2}\|P_aC_aa\|^4)}{s_a^2\|a-\hat a\|^2} \right)
 	\label{eq:proof_firstterm}
 \end{align}
 In order to calculate the limit of this expression, we first consider
 \begin{align*}
 	\lim_{a\to \hat a} s_a \|a-\hat a\|
 	= \lim_{a\to \hat a} \sum_{i=1}^N \|a-\hat a\| \frac{(a^\tT y_i)^2}{\|P_ay_i\|}
 	=\lim_{a\to \hat a} \frac{\|a-\hat a\|}{\|P_a\hat a\|} \frac{(a^\tT y_k)^2}{\|y_k\|}
 	\end{align*}
 and since
 \[
 	\lim_{a\to \hat a} \frac{\|a-\hat a\|}{\|P_a\hat a\|} 
	= \lim_{a\to \hat a} \frac{2(1-\langle a,\hat a \rangle)}{(1-\langle a,\hat a\rangle)(1+\langle a,\hat a\rangle)} 
	= \lim_{a\to \hat a} \frac{2}{1+\langle a,\hat a \rangle}
	=1,
 \]
finally
$$
\lim_{a\to \hat a} s_a \|a-\hat a\| = \|y_k\|.
$$
The remainder of the Taylor approximation converges to zero as $\|P_aC_aa\|$ is bounded from above, 
while $s_a$ goes to infinity. 
Together with \eqref{eq:proof_firstterm} this gives the limit of the first term,
 \begin{align}
 	\lim_{a\to \hat a} \frac{1}{\|T_a\|^2} \left( 1 + \frac{(1-\langle a,\hat a \rangle)\|P_aC_aa\|^2 
	+ \mathcal{O}(\frac{1}{s_a^2}\|P_aC_aa\|^4)}{s_a^2\|a-\hat a\|^2} \right)
 	= 1
 \end{align}
 For the second term in \eqref{eq:proof_terms} we calculate
 \begin{align}
 	L =&\lim_{a\to \hat a}\frac{2\langle a-\|T_a\|\hat a,\frac{1}{s_a}P_a C_a a \rangle 
	+ \frac{1}{s_a^2}\|P_aC_aa\|^2}{\|T_a\|^2\|a-\hat a\|^2}\\
 	&= \lim_{a\to \hat a} \frac{-2\|T_a\|s_a\langle \hat a,G_{a,k} 
	+ a^\tT y_k\frac{P_ay_k}{\|P_ay_k\|} \rangle + \|G_{a,k} + a^\tT y_k\frac{P_ay_k}{\|P_ay_k\|}\|^2}{\|T_a\|^2s_a^2\|a-\hat a\|^2} \label{num}
	\end{align}
	Now it is straightforward to check that 
	$$
	\lim_{a\to \hat a} \big( \sum_{i \not = k} \frac{(a^\tT y_i)^2}{\|P_a y_i\|} \big) \langle \hat a,G_{a,k} + a^\tT y_k\frac{P_ay_k}{\|P_ay_k\|} \rangle  = 0
	$$
	so that by definition of $s_a$ the  term \eqref{num} becomes
	\begin{align*}
	 	L\coloneqq &\lim_{a\to \hat a} \frac{-2\|T_a\|
		\frac{(a^\tT y_k)^2}{\|P_ay_k\|} \left\langle \hat a,G_{a,k} 
	+ a^\tT y_k\frac{P_ay_k}{\|P_ay_k\|} \right\rangle + \|G_{a,k} + a^\tT y_k\frac{P_ay_k}{\|P_ay_k\|}\|^2}{\|T_a\|^2s_a^2\|a-\hat a\|^2}.
	\end{align*}
	Using that $G_{a,k} = P_a C_{a,k} a$, $y_k = \hat a \|y_k\|$ and $P_a$ is an orthogonal projector, we can simplify  
	\begin{align*}
	\frac{(a^\tT y_k)^2}{\|P_ay_k\|} \langle \hat a,G_{a,k} \rangle  
	&=  
	\frac{(a^\tT y_k)^2}{\|y_k\|} \left \langle \frac{P_a y_k}{\|P_a y_k\|},G_{a,k} \right\rangle,
	\\
	\frac{(a^\tT y_k)^2}{\|P_ay_k\|} \left \langle \hat a,a^\tT y_k \frac{P_ay_k}{\|P_ay_k\|} \right\rangle
	&= 
	\frac{(a^\tT y_k)^3}{\|y_k\|} \left\langle \frac{P_ay_k}{\|P_ay_k\|},\frac{P_ay_k}{\|P_ay_k\|} \right\rangle = \frac{(a^\tT y_k)^3}{\|y_k\|},
	\end{align*}
	so that
	\begin{align*}
	 	L
 	&= \lim_{a\to \hat a} \frac{\|G_a,k\|^2 + (a^\tT y_k)^2 - 2\|T_a\| \frac{(a^\tT y_k)^3}{\|y_k\|} + a^\tT y_k
	\langle \frac{P_ay_k}{\|P_ay_k\|}, G_{a,k} \rangle \left( -2\|T_a\|\frac{a^\tT y_k}{\|y_k\|} + 2 \right)}{\|T_a\|s_a^2\|a-\hat a\|^2}.
 \end{align*}
 As $\langle \frac{P_ay_k}{\|P_ay_k\|}, G_{a,k} \rangle$ is bounded, 
$\lim_{a\to \hat a} a^\tT y_k = \|y_k\|$ and $\lim_{a\to \hat a} \|T_a\|=1$ 
we get
 \[
 	\lim_{a\to \hat a}\frac{2\langle a-\|T_a\|\hat a,\frac{1}{s_a}P_aC_aa \rangle + \frac{1}{s_a^2}\|P_aC_aa\|^2}{\|T_a\|^2\|a-\hat a\|^2} 
	= \frac{\|G_{\hat a,k}\|^2 - \|y_k\|^2}{\|y_k\|^2}.
 \]
 Plugging the results into \eqref{eq:proof_terms} yields the assertion
 \[
 	\lim_{a\to \hat a} \frac{\|T(a) - \hat a\|^2}{\|a-\hat a\|^2} 
	= 1 + \frac{\|G_{\hat a,k}\|^2 - \|y_k\|^2}{\|y_k\|^2} = \frac{\|G_{\hat a,k}\|^2}{\|y_k\|^2}.
 \]
\end{proof}

Finally, we need the Kurdyka–Łojasiewicz property of functions \cite{ABRS2010}:
The function $f\colon \R^d \to \R \cup \{+\infty\}$ with Fr\'echet limiting subdifferential $\partial f$, see \cite{mord2006},
is said to have the \emph{Kurdyka–Łojasiewicz (KL) property} at $x^* \in \dom \partial f$ 
if there exist $\eta \in (0, +\infty)$, 
a neighborhood $U$ of $x^*$ and a continuous concave function $\phi\colon [0,\eta) \to \R_{\geq 0}$ such that
	\begin{enumerate}
		\item $\phi(0) = 0$,
		\item $\phi$ is $C^1$ on $(0,\eta)$,
		\item for all $s\in(0,\eta)$ it holds $\phi'(s)>0$,
		\item for all $x\in U\cup [f(x^*)< f <f(x^*) + \eta]$, 
		the Kurdyka–Łojasiewicz inequality  $\phi'(f(x)-f(x^*)) \dist(0,\partial f(x)) \geq 1$ holds true.
	\end{enumerate}
	A proper, lower semi-continuous (lsc) function which satisfies the KL property at each point of $\dom \partial f$ is called 
	\emph{KL-function}.
	Typical examples of KL functions are semi-algebraic functions.
	Fundamental works on this subject go back to Łojasiewicz \cite{Lo1963} and Kurdyka \cite{Kur1998}.

The next theorem was proved by Bolte, Attouch and Svaiter \cite[Theorem 2.9]{ABSF13}.

\begin{theorem}\label{KLconv}
	Let $f\colon \R^d \to \R \cup \{\infty\}$ be a KL function. 
	Let $\{ x^{(r)}\}_{r \in \N}$ be a sequence  which fulfills the following conditions:
	\begin{enumerate}
		\item[C1.] There exists $K_1>0$ such that $f(x^{(r+1)}) - f(x^{(r)}) \leq -K_1 \Vert x^{(r+1)} - x^{(r)} \Vert^2$ for every $r\in \N$.
		\item[C2.] There exists $K_2>0$ such that for every $r\in \N$ 
		there exists $w_{r+1} \in \partial f(x^{(r+1)})$ with $\Vert w_{r+1} \Vert \leq K_2 \Vert x^{(r+1)} - x^{(r)} \Vert$.
		where $\partial f$ denotes the Fr\'echet limiting subdifferential of $f$ \cite{mord2006}.
		\item[C3.] There exists a convergent subsequence $\{ x^{(r_j)} \}_{j \in \N}$ with limit $\hat x$ and $f(x^{(r_j)}) \to f(\hat x)$.
	\end{enumerate}
	Then the whole sequence $\{ x^{(r)} \}_{r \in \N}$ converges to $\hat x$ 
	and $\hat x$ is a critical point of $f$ in the sense that $0 \in \partial f(x)$.
	Moreover the sequence has finite length, i.e.,
	\[\sum_{r=0}^\infty \Vert x^{(r+1)} - x^{(r)} \Vert < \infty.\]
\end{theorem}

Clearly, if $f$ is differentiable at $x$, 
then $x$ is a critical point of $f$, if and only if $\nabla f(x) = 0$.
We will only need this case.

Similar arguments as used in the proof of the above theorem lead to the next corollary, see \cite[Corollary 2.7]{ABSF13}.

\begin{corollary}\label{KLBound}
	Let $f\colon \R^d \to \R \cup \{+\infty\}$ be a proper, lsc function 
	which satisfies the KL property at $x^*$. 
	Denote by $U$, $\eta$ and $\phi$ the objects appearing in the definition of the KL function.
	Let $\delta, \rho > 0$ be such that $B(x^*,\delta)\subset U$ with $\rho \in (0,\delta)$. 
	Consider a finite sequence $x^{(r)}$, $r= 0, \dots,n$, 
	which satisfies the Conditions C1 and C2 of Theorem \ref{KLconv} and additionally
	\begin{enumerate}
		\item[C4.] $f(x^*) \leq f(x^{(0)}) < f(x^*) + \eta$,
		\item[C5.]  $\Vert x^* - x^{(0)} \Vert + 2\sqrt{\frac{f(x^{(0)}) -f(x^*)}{K_1}} + \frac{K_2}{K_1} \phi(f(x^{(0)}) - f(x^*)) \leq \rho$.
	\end{enumerate}
	If for all $r=0,\dots, n$ it holds 
	$$x^{(r)} \in B(x^*,\rho) \quad \Longrightarrow \quad  x^{(r+1)} \in B(x^*,\delta) \; \mathrm{and} \;  f(x^{(r+1)})\geq f(x^*),$$
	then $x^{(r)} \in B(x^*,\rho)$ for all  $r=0,\dots, n+1$.
\end{corollary}

Now we can prove our main convergence theorem.

\begin{theorem}
  The sequence $\{a^{(r)}\}_{r}$ generated by Algorithm \ref{alg:onedir} converges to a critical point of $E$.
\end{theorem}

\begin{proof}
		If the sequence is finite, the claim follows from Lemma~\ref{alg_stop}.
		Assume that the algorithm produces an infinite sequence.
		Since the sequence $(a^{(r)})_{r\in\NN}$ on $\mathbb S^{d-1}$ is bounded,
		there exists a convergent subsequence $(a^{(r_j)})_{j\in\NN}$ with $\lim_{j\to\infty} a^{(r_j)} = \hat a$. 
		Possibly, there exist multiple accumulation points and we distinguish two cases.
		\begin{enumerate}
			\item
			First assume that no accumulation point is in the anchor set.
			It is easy to verify that the function $E$ is semi-algebraic on $\mathbb R^{d}$ and hence fulfills the KL property.
			We will verify that $\{a^{(r)}\}_{r}$ fulfills the remaining conditions C1 and C2 from Theorem~\ref{KLconv}.
			From the proof of Lemma~\ref{abstieg}, Case 1, we get
			\[E(a^{(r)}) - E(a^{(r+1)})  
			\geq 
			\sum_{i=1}^N  \frac{\Vert  \langle a^{(r)},y_i\rangle a^{(r)} - \langle a^{(r+1)},y_i \rangle a^{(r+1)}  \Vert^2}{2\Vert P_{a^{(r)}}y_i\Vert}.\]
			Further, it holds $\Vert P_{a^{(r)}} y_i \Vert \leq \Vert y_i \Vert \leq \max_{i=1,\dots,N} \Vert y_i \Vert < \infty$ and
			there exists $m>0$ such that
			\[\min_{{a\in\mathrm{span}(Y)}\atop{\Vert a \Vert = 1}} \max_{i=1,\dots,N}\vert \langle a,y_i\rangle \vert =  \min_{{a\in\mathrm{span}(Y)}\atop{\Vert a \Vert = 1}} \|Y^\tT a \|_\infty \ge m.\]
			Using that $a^{(r)}\in \mathrm{span}(Y)$ for all $r$ and
			$\lim_{r \to \infty} \Vert a^{(r+1)} - a^{(r)} \Vert = 0$ by Lemma~\ref{lem:conv_1}, 
			we can find $i \in \{1,\ldots,N\}$ such that 
			$\vert \langle a^{(r)},y_i\rangle \vert > \frac{m}{2}$, 
			$\vert \langle a^{(r+1)},y_i\rangle \vert > \frac{m}{2}$ 
			and both scalar products have the same sign for $r$ large enough.
			Hence we can estimate
			\[
			E(a^{(r)}) - E(a^{(r+1)})  \geq C \left\Vert a^{(r)}  - a^{(r+1)} \frac{\langle a^{(r+1)},y_i \rangle}{\langle a^{(r)},y_i\rangle} \right\Vert^2,
			\quad C > 0,
			\]
			where w.l.o.g $\frac{\langle a^{(r+1)},y_i \rangle}{\langle a^{(r)},y_i\rangle}\geq 1$.
			Using the projection onto the sphere, we can finally estimate
			\[E(a^{(r)}) - E(a^{(r+1)})  \geq C \Vert a^{(r)}  - a^{(r+1)}\Vert^2.\]
			
			Next we check the second condition C2.
			Since $\lim_{r \to \infty} \Vert a^{(r+1)} -a^{(r)} \Vert = 0$ 
			and none of the 
			$\pm y_i/\|y_i\|$, $i=1,\ldots,N$, is an accumulation point, 
			we can find open balls $B_i$ around every $y_i$ 
			such that for all $r$ large enough we have 
			$\overline{a^{(r)} a^{(r+1)}} \subset \Omega \coloneqq \R^d\setminus \bigcup_{i=1}^N B_i$.
			The function $E$ is smooth on an open set containing the compact set $\Omega$ 
			and hence there exists $C>0$ such that
			\[\Vert \nabla E (a^{(r+1)}) - \nabla E(a^{(r)}) \Vert \leq C\Vert a^{(r+1)} - a^{(r)} \Vert\]
			for all $r$ large enough.
			Further, note that the sequence $s_{a^{(r)}}$ is bounded from above on $\Omega$ 
			which implies
			\[
			\Vert \nabla E (a^{(r+1)})\Vert 
			\leq 
			\tilde C \left(\Vert a^{(r+1)} - a^{(r)} \Vert +\left\Vert \frac{\nabla E(a^{(r)})}{s_{a^{(r)}}} \right\Vert \right).
			\]
			Using the iteration law $a^{(r+1)} = \Pi_{\Ss^{d-1}}(a^{(r)} - \frac{\nabla E(a^{(r)})}{s_{a^{(r)}}})$ together with the fact that 
			$\nabla E(a^{(r)})$ is in the tangential plane of $\mathbb S^{d-1}$ at $a^{(r)}$, we get by
			the law of sines, see Fig.~\ref{fig:law_sines},
			\[
			\frac{\Vert a^{(r+1)} - a^{(r)} \Vert s_{a^{(r)}}}{\Vert \nabla E(a^{(r)}) \Vert} 
			\geq 
			\sin\left(\frac{\pi}{2}- \angle(a^{(r)} a^{(r+1)})\right),\]
			where the right hand side converges to one since $\angle(a^{(r)} a^{(r+1)})$ gets arbitrary small.
			Hence, the right hand side is larger than 
			$\frac{1}{2}$ for $r$ large enough 
			and we can estimate
			\[\Vert \nabla E (a^{(r+1)})\Vert \leq 3 \tilde C \Vert a^{(r+1)} - a^{(r)} \Vert.\]
			Now, by Theorem~\ref{KLconv}, only one accumulation point exists which is also a critical point.

			\begin{figure}
			\begin{center}
			\includegraphics[width=0.4\textwidth, angle=0]{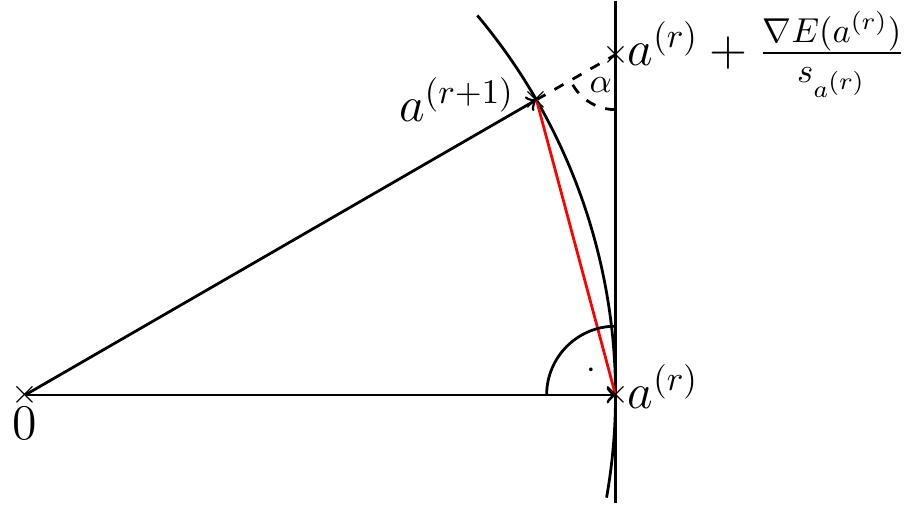}
			\end{center}
			\caption{Sketch law of sines for our setting.} \label{fig:law_sines}
			\end{figure}
			
			\item It remains to examine the case that some accumulation point is an anchor point
			$\hat a $ to the vertices $y_k$, $k \in \mathcal{K}$.
			Assume that there exists another accumulation point. 
			Then, by Lemma~\ref{lem:conv_1}, there exists 
			an accumulation point $\tilde a$ which is not an anchor point.
			We can find a ball $B(\tilde a, R)$ around $\tilde a$ which has positive distance to all anchor points. 
			Next, for all the iterates $a^{(r)} \in B( \tilde a, \frac{R}{2} )$ and  $r$ large enough
			we can reproduce step one of the proof 
			to show that C1 and C2 are fulfilled.
			Be the continuity of $f$ and $\phi$, see also the proof of \cite[Theorem~2.9]{ABSF13}, 
			we can choose a ball $B(\tilde a, \delta) \subset B( \tilde a, \frac{R}{2} )\cap U$ 
			(where $U$ is from the definition of the KL property), $\rho \in (0,\delta)$ 
			and a starting iterate $a^{(r_0)} \in B(\tilde a, \rho)$ which satisfies C4  and C5
			from Corollary~\ref{KLBound}.
			Since $\lim_{r \to \infty} \Vert a^{(r+1)} - a^{(r)} \Vert = 0$ 
			and $\tilde a$ is an accumulation point, we can choose $r_0$ such that
			\[a^{(r)} \in B(\tilde a,\rho) \implies a^{(r+1)} \in B(\tilde a,\delta), \; E(a^{(r+1)})\geq E(\tilde a)\]
			for all $r \geq r_0$.
			Either all iterates after $a^{(r_0)}$ are in $B(\tilde a, \rho)$
			or there is a finite sequence $a^{(r_0)},a^{(r_0+1)} ,\ldots, a^{(r_n)}$ such that
			$a^{(r_{n}+1)}$ is the first element outside $B(\tilde a, \rho)$.
			But then, by Corollary~\ref{KLBound}, also the iterate $a^{(r_n+1)}$ is inside $B(\tilde a,\rho)$ and hence all iterates stay in $B(\tilde a,\rho)$ 
			which is an contradiction.
			Consequently, the whole sequence converges to the anchor point $\hat a$.
			
			It remains to show that the anchor point is critical. By Lemma \ref{lem:conv_2} we know that
			\[
			\lim_{r\to\infty} \frac{\|T(a^{(r)}) - \hat a\|}{\|a^{(r)} - \hat a\|} = \frac{\|G_{\hat a,\mathcal{K} }\|}{\sum_{k \in \mathcal{K}} \|y_k\|}.
			\]
			If  $\hat a$ is not a critical point, i.e.~$\|G_{\hat a,\mathcal{K}}\|>{\sum_{k \in \mathcal{K}} \|y_k\|}$, 
			then the sequence cannot converge to $\hat a$, which is a contradiction.		
				\end{enumerate}
\end{proof}

At this point it should be mentioned that Algorithm \ref{alg:onedir} 
may converge to a local minimum as our functional is non-convex. 
Performing the algorithm multiple times with random initialization $a^{(0)}$ 
and comparing the function values of the results increases the probability to reach a global minimizer.
The number of local minimizers and how pronounced they are, depends on the data. 
In general, with fewer data points and more extreme outliers, 
we tend to get more pronounced local minima. 
However, in most applications and also in the numerical part of this paper, 
this is not an issue as a high number of data points is available. 

\section{Remarks on the Offset}\label{sec:offset}
Finally, we want to address briefly the issue of choosing a suitable offset for the robust PCA model.

As already mentioned in the introduction, in classical PCA, solving
\begin{equation}
\argmin_{A \in \mathbb S_{d,K}, b \in \mathbb R^d} \sum_{i=1}^N \min_{t \in \mathbb R^K} \|A\, t + b - x_i \|^2
= \argmin_{A \in \mathbb S_{d,K}, b \in \mathbb R^d} \sum_{i=1}^N \|P_A(x_i -b)\|^2,
\end{equation}
leads to the unique affine subspace
$$
\{\hat A t + \hat b: t \in \mathbb R^K\},
$$
where $\hat b \in \mathbb R^d$ can be chosen as mean (bias) $\bar b := \frac{1}{N}(x_1 + \ldots + x_N)$ of the data.
For the robust setting, we have assumed so far that the offset $\hat b$ is given, e.g., as geometric median of the data.
However, the problem
\begin{equation}\label{prob_tag}
\argmin_{A \in \mathbb S_{d,K}, b \in \mathbb R^d} \sum_{i=1}^N \min_{t \in \mathbb R^K} \|A\, t + b - x_i \|
= 
\argmin_{A \in \mathbb S_{d,K}, b \in \mathbb R^d}  \sum_{i=1}^N \|P_A(x_i -b)\|
\end{equation}
has in general not the geometric median as correct offset as we will see in the following.

\begin{lemma}\label{lem:line}
Let $x_i \in \mathbb R^2$, $i =1,\ldots,N$.
Then there exists a minimizing pair
\[
(\hat a, \hat b) \in \argmin_{a \in \mathbb S^{1}, b\in\R^2}\sum_{i=1}^N \min_{t \in \mathbb R} \|a t + b - x_i\|
\]
such that the line $g(t) := \hat at + \hat b$ passes through two of the points. If $N$ is odd, then the minimizing line always passes through two points.
\end{lemma}

\begin{proof}
Assume that $g$ is an optimal line which does not go through any of the points.
Let $N_l$, resp. $N_r$ be the number of points on the left, resp., right hand side of $g$.
Then, shifting the line by $\delta >0$
into the direction of the left, resp. right nearest point changes the distance sum by 
$(N_r-N_l)\delta$, resp.  $(N_l-N_r)\delta$. If $N_l \not = N_r$, 
then one of the new distance sums becomes smaller than the original minimal one. 
Hence, $N_l = N_r$, so that one point has to be on the line if $N$ is odd
and there is a line with smallest distance sum going through one point if $N$ is even. W.l.o.g., let $g$ go through $x_N$.
Then, choosing $b=x_N$ we have to show that $g$ goes through a second point.
Taking polar coordinates $y_i:= x_i - x_N = c_i \e^{i \gamma_i}$, $i\in \{1,\ldots,N-1\}$,
and
$a = \e^{i\alpha}$, the distance sum becomes
\[
 \sum_{i=1}^{N-1} \min_{t_i} |t_i \e^{i\alpha} - c_i \e^{i \gamma_i}| 
= \sum_{i=1}^{N-1} \min_{t_i} |t_i - c_i \e^{i (\gamma_i - \alpha)}|
= \sum_{i=1}^{N-1} |c_i \sin(\alpha - \gamma_i)| =: \varphi(\alpha).
\]
If $y_i \not \in g$ for all $i=1,\ldots,N-1$, then $\varphi$ is smooth and
\[
\varphi''(\alpha) = -\varphi(\alpha) < 0
\]
so that $\alpha$ cannot be a local minimizer. Consequently, at least one more $x_i$ must lie on $g$.
\end{proof}

Using the decomposition of $A \in \mathbb S_{d,d-1}$ into Givens rotation matrices, 
the claim can be generalized to hyperplanes of dimension $d-1$ having minimal 
Euclidean distance from data in $\mathbb R^d$, $d \ge 2$,
see \cite{schneck2018}.
However, it would be interesting if in this case also $d-1$ data points 
can lie within the minimizing hyperplane instead of just two of them.
\\

Based on the lemma, the following example shows that the geometric median
is in general not in the solution set of \eqref{prob_tag}.

\begin{example} \label{ex:tag}
Let $x_i \in \mathbb R^2$, $i=1,2,3$ span a triangle with sides
$s_1 = \|x_2-x_3\|$, $s_2 = \|x_1-x_3\|$, $s_3 = \|x_1-x_2\|$,
where $s_1 \le s_2 < s_3$
and angles smaller than $120^\circ$.
By Lemma \ref{lem:line}, 
the line having minimal Euclidean distance from the three
points has to go through two points.
Since the height $h_i$ at side $s_i$, $i=1,2,3$ fulfills 
$$h_i = \frac{2}{s_i}\left(s(s-s_1)(s-s_2)(s-s_3)\right)^\frac12, \quad s := \frac13(s_1+s_2+s_3),$$
we conclude that the line must go through $x_1$ and $x_2$ and has distance $h_3$ from $x_3$.
On the other hand, it is easy to check (and known) that the geometric median 
of the data points is the so-called Steiner point  
from which the points can be seen under an angle of $120^\circ$.
Clearly, the minimizing line does not pass the Steiner point.
\end{example}

\section{Numerical Examples} \label{sec:numerics}
In this section, we present various numerical examples. 
In particular, we compare our Algorithm \ref{alg:onedir} (with the geometric median $b$)
with standard PCA and the following methods:
\begin{itemize}
\item[i)] \textbf{PC-L1}: 
the greedy algorithm for minimizing \eqref{pcd_robust_3} proposed by Kwak \cite{kwak2008principal}.
As $b$ we used the geometric median computed by Weiszfeld's algorithm \ref{alg:onedir}.	

\item[ii)] \textbf{TRPCA}: 
the trimmed PCA of Podosinnikova et al. \cite{podosinnikova2014robust} with default parameters, i.e., the lower bound on the number of true observations is set to $\frac N2$ and the number of random restarts is  $10$.
Here, $b$ equals the mean of a certain subset of the given data determined within the algorithm.
\end{itemize}

	\subsection{Image Sequence with Slightly Varying Background}
	
		\begin{figure}
		\centering
		\begin{subfigure}{0.23\textwidth}
			\includegraphics[width=1\textwidth]{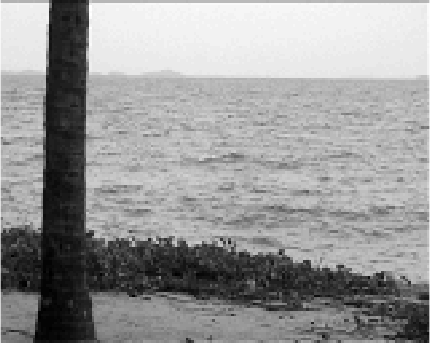}
			\caption{}\label{subfig:water_1}
		\end{subfigure}
		\begin{subfigure}{0.23\textwidth}
			\includegraphics[width=1\textwidth]{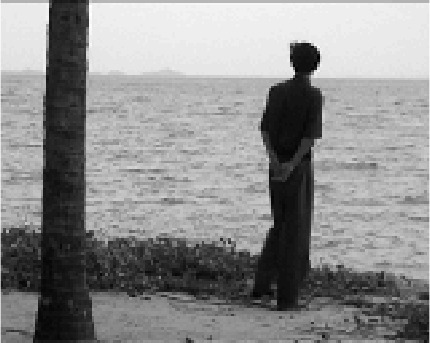}
			\caption{}\label{subfig:water_580}
		\end{subfigure}
		\begin{subfigure}{0.23\textwidth}
			\includegraphics[width=1\textwidth]{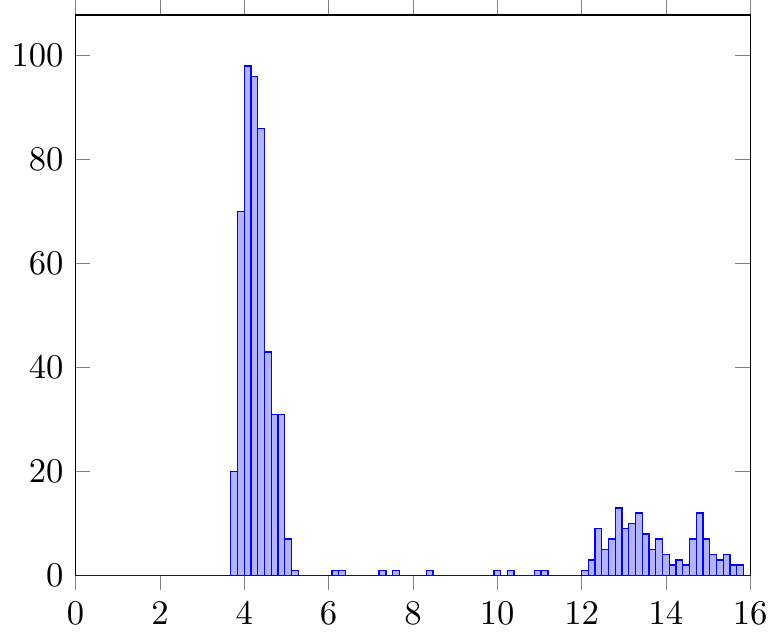}
			\caption{}\label{subfig:water_median_dist_hist}
		\end{subfigure}
		\begin{subfigure}{0.23\textwidth}
			\includegraphics[width=1\textwidth]{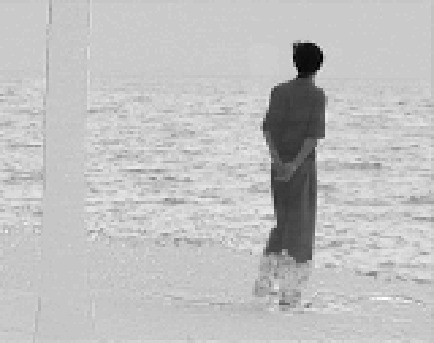}
			\caption{}\label{subfig:water_fg}
		\end{subfigure}
		\caption{Processing of the \emph{water front data set}. 
		\protect\subref{subfig:water_1} first frame showing the background of grass and sea 
		which slightly varies over the time, 
		\protect\subref{subfig:water_580} frame $580$ is an outlier frame with the person present, 			
		\protect\subref{subfig:water_median_dist_hist} 
		histogram of the distances of the frames 
		to their the geometric median, 
		\protect\subref{subfig:water_fg} difference between frame $580$ and the geometric median.}\label{fig:water}
	\end{figure}	
	
	We consider an image sequence with slightly varying background as it was used for object detection in various papers.
	The \emph{water front data set}, see Fig.~\ref{fig:water}, 
	was originally considered in \cite{li2004statistical}. 
	It was used for performance comparisons with several robust PCA methods 
	including those of Candes et al. \cite{candes2011robust} 
	in the context of object detection in \cite{podosinnikova2014robust}, 
	where TRPCA outperformed the other methods.
	The data set consists of $633$ frames of size $128\times 160$
	of a scenery with water and grass as background. 
	Beginning with frame $481$ a person walks into the scene, 
	which we consider as "outlier" frames in the data set. 
	We aim to detect the frames with the person present, 
	and then to separate background (scenery) and foreground (person).
	It turns out this can be achieved simply by thresholding the Euclidean distances of the vectorized data $x_i\in\R^{20480}$ to their geometric median. 
	The frames with the person in them can be detected from the histogram  in Fig.~\ref{subfig:water_median_dist_hist}. 
	More precisely, all frames with distance larger than $6$ 
	can be considered as outliers which exactly matches the frames with the person present. 
	The foreground in these images can then be extracted as the difference image to the geometric median 
	and subsequent pixelwise thresholding. 
	The difference image for one frame is given in Fig.~\ref{subfig:water_fg}.
	\\
	
In order to make the task more challenging   
and simulate a gradual change in lighting conditions, 
we alter the data as follows. 
Given the points $x_i$, $i=1,\ldots,633$, we created new data
	\begin{equation} \label{fake}
	\tilde x_i \coloneqq \frac 34 x_i + \frac{i}{8\cdot 633}(\boldsymbol 1 +x_i).
	\end{equation}
Here, outlier frames cannot be found by the previous method since the
distance of the frames from their geometric median varies by construction, 
see Fig.~\ref{fig:water_scaled_histo} (left). 
But a model with line fitting, i.e. with $K=1$,
is suitable by the construction of the data. 
Fig.~\ref{fig:water_scaled_histo} depicts the histogram of the distances of the frames $\tilde x_i$ from the line
generated by the standard PCA and by the residual minimizing robust PCA, respectively. 
The outliers can be better separated by the residual minimizing robust PCA 
as the frames belonging to the peak between $4$ and $5$ 
are less likely to be wrongfully mislabeled as outliers.
 
Finally, Fig.~\ref{fig:water_scaled_rec_res} shows the foreground--background separation in frame $i=580$ by various methods.
We show the projected data (background)
$x_{i,\mathrm{rec}} = a_1 a_1^\tT (x_i - b) + b$ (left)
and the residual (person)
$x_{i,\mathrm{res}} = x_i - x_{i,\mathrm{rec}}$, $i=580$. 
In the background and foreground of standard PCA, 
artifacts can be clearly seen at positions where the person rests for a longer time.
The PCA-L1 and our approach appear to be  more robust here, 
but the artifacts are more pronounced in the PCA-L1.
TRPCA with several restarts gives the best results.
	
	\begin{figure}
				\begin{subfigure}{0.32\textwidth}
			\includegraphics[width=1\textwidth]{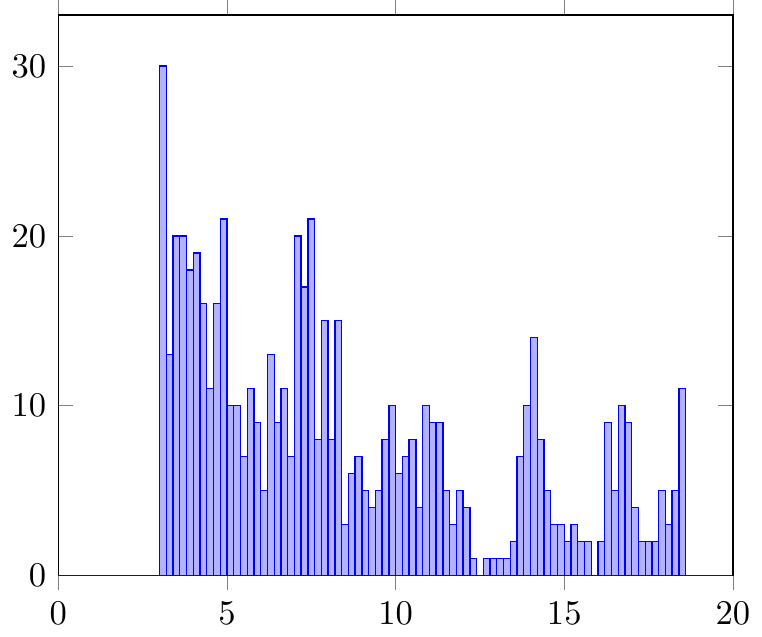}
	  		\end{subfigure}
		\begin{subfigure}{0.32\textwidth}
			\includegraphics[width=1\textwidth]{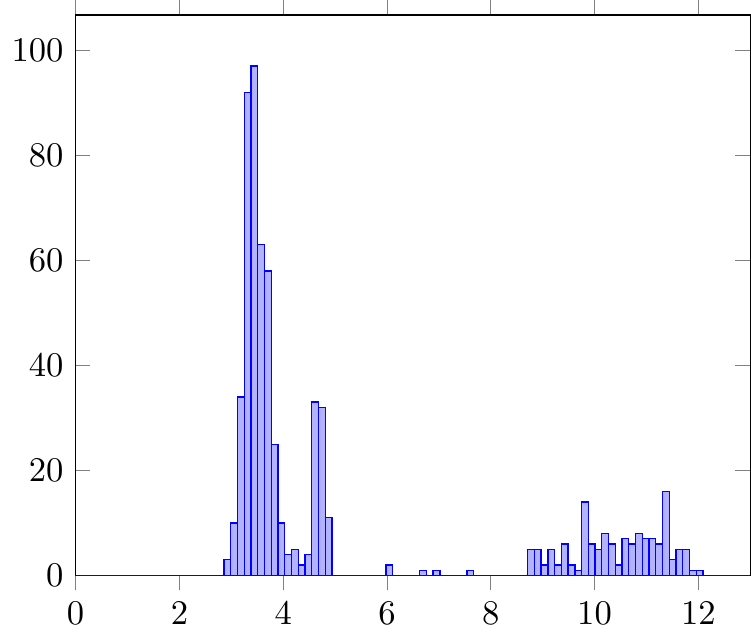}
		\end{subfigure}
			\begin{subfigure}{0.32\textwidth}
			\includegraphics[width=1\textwidth]{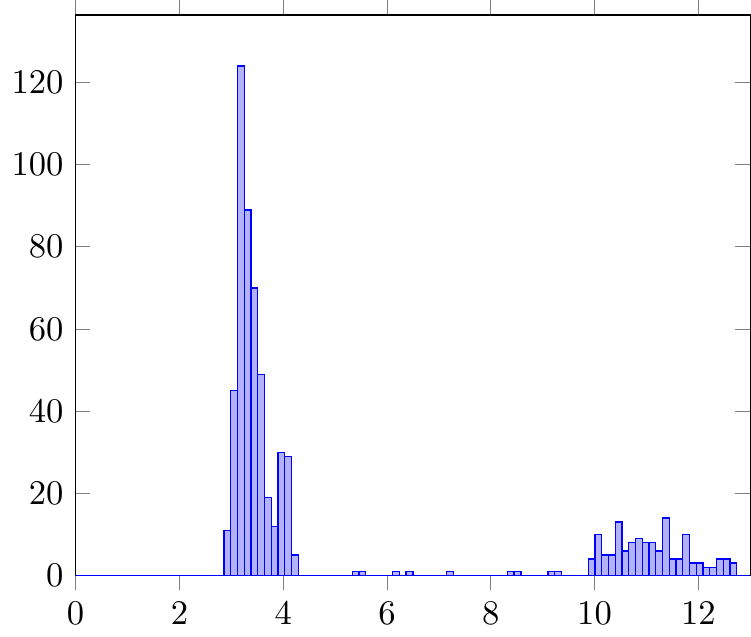}
			\end{subfigure}
		\caption{Distance histograms for the \emph{water front data set} modified by \eqref{fake}. 
		Left: distance of the frames to their median. 
		Middle: distance of the frames to the line fitted with standard PCA. 
		Right: distance of the frames to the line fitted with our approach.}
		\label{fig:water_scaled_histo}
	\end{figure}
	
	\begin{figure}
		\centering
		\begin{subfigure}{0.23\textwidth}
			\includegraphics[width=1\textwidth]{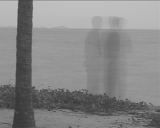}				
		\end{subfigure}
		\begin{subfigure}{0.23\textwidth}
			\includegraphics[width=1\textwidth]{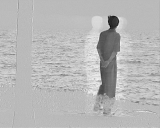}			
		\end{subfigure}\hspace{0.4cm}
		\begin{subfigure}{0.23\textwidth}
			\includegraphics[width=1\textwidth]{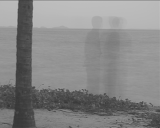}		
		\end{subfigure}
		\begin{subfigure}{0.23\textwidth}
			\includegraphics[width=1\textwidth]{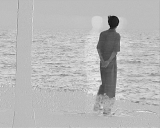}	
		\end{subfigure}\\ \vspace{0.3cm}
		\begin{subfigure}{0.23\textwidth}
			\includegraphics[width=1\textwidth]{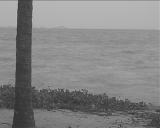}		
		\end{subfigure}
		\begin{subfigure}{0.23\textwidth}
			\includegraphics[width=1\textwidth]{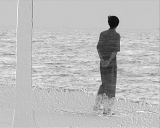}
		\end{subfigure}\hspace{0.4cm}
		\begin{subfigure}{0.23\textwidth}
			\includegraphics[width=1\textwidth]{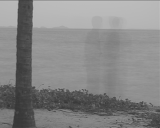}
			\end{subfigure}
		\begin{subfigure}{0.23\textwidth}
			\includegraphics[width=1\textwidth]{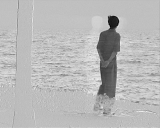}
				\end{subfigure}
		\caption{Projections on one dimensional subspace (background) 
		and residuals (foreground) of different approaches for the \emph{water front data set}. 
		Top left: standard PCA. 
		Top right: PCA-L1. 
		Bottom left: TRPCA.
		Bottom right: our approach.}
		\label{fig:water_scaled_rec_res}
	\end{figure}
		
	\subsection{Face reconstruction}
	For images of faces in the same pose but with different lighting, it may be assumed that they lie in a low dimensional subspace \cite{epstein1995}. 
	Thus standard PCA is a suitable tool to reduce the dimensionality of such data for classification and other tasks. 
	In practice, however, some of the face images may be occluded resulting in  outliers within the data. 
	Here, robust PCA methods appear to be more useful. 
	We test the performance of various approaches 
	with the cropped Extended Yale Face database B \cite{LHK2005}. 
	There are $58$ images of size $168\times 192$ of which $12$ 
	were altered with a $50\times 50$ square patch of noise at a random position, 
	see the left image in Fig.~\ref{fig:face_data} for an example.
	
	For noiseless face image data, standard PCA projection onto a subspace of dimension $K=5$ gives good results as shown on the right of Fig.~\ref{fig:face_data}.
	
	In Fig.~\ref{fig:face_rec} the projection of the noisy data on the subspace obtained by various approaches are shown. 
	As expected the noisy patches can be clearly seen in the reconstructions by standard PCA.
	Surprisingly, the result of  PCA-L1 looks worse than standard PCA, 
	as the influence of the noisy patches is even worse. 
	The results of TRPCA with several restarts are very similar to those of the standard PCA of the noiseless data 
	except of the right eye in the second image which appears to be too dark.
	This suggests that the algorithm successfully 
	excluded the outliers and calculated the principal components from a part of the noiseless data.
  	However, it should be mentioned that TRPCA sometimes fails to detect the outliers 
	as it depends on the initial values of the random restarts.
	The results of residual minimizing robust PCA demonstrate the robustness of the method to outliers
	although slight artifacts are still visible.
	
		\begin{figure}
		\centering
		\begin{subfigure}{0.23\textwidth}
			\includegraphics[width=1\textwidth]{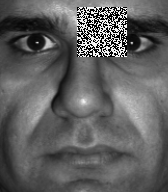}
		\end{subfigure}
		\begin{subfigure}{0.23\textwidth}
			\includegraphics[width=1\textwidth]{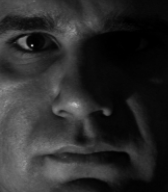}
		\end{subfigure}\hspace{0.4cm}
		\begin{subfigure}{0.23\textwidth}
			\includegraphics[width=1\textwidth]{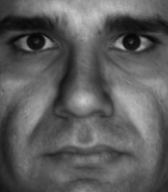}
		\end{subfigure}
		\begin{subfigure}{0.23\textwidth}
			\includegraphics[width=1\textwidth]{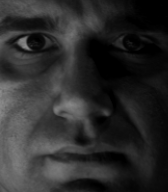}
		\end{subfigure}
		\caption{Standard PCA with $K=5$ for noiseless face data.
		Left two images: 
		Different face images, where a possible corruption is exemplified by the first one.
		Right two images: projections of the noiseless images onto 
		the $5$ dimensional subspace calculated by standard PCA computed from the 
		 noiseless data set.}
		\label{fig:face_data}
	\end{figure}
	
	\begin{figure}
		\centering
		\begin{subfigure}{0.23\textwidth}
			\includegraphics[width=1\textwidth]{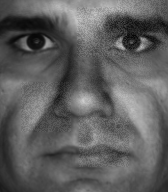}
					\end{subfigure}
		\begin{subfigure}{0.23\textwidth}
			\includegraphics[width=1\textwidth]{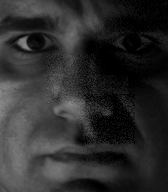}
				\end{subfigure}\hspace{0.4cm}
		\begin{subfigure}{0.23\textwidth}
			\includegraphics[width=1\textwidth]{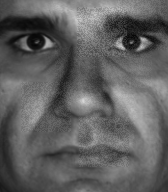}
				\end{subfigure}
		\begin{subfigure}{0.23\textwidth}
			\includegraphics[width=1\textwidth]{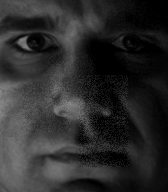}
					\end{subfigure}\\ \vspace{0.4cm}
		\begin{subfigure}{0.23\textwidth}
			\includegraphics[width=1\textwidth]{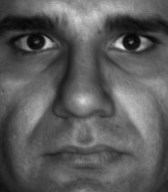}
				\end{subfigure}
		\begin{subfigure}{0.23\textwidth}
			\includegraphics[width=1\textwidth]{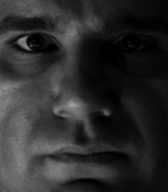}
					\end{subfigure}\hspace{0.4cm}
		\begin{subfigure}{0.23\textwidth}
			\includegraphics[width=1\textwidth]{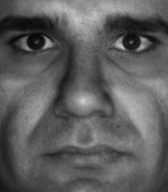}
					\end{subfigure}
		\begin{subfigure}{0.23\textwidth}
			\includegraphics[width=1\textwidth]{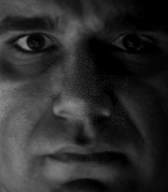}
			\end{subfigure}
		\caption{PCA methods for the noisy face data.
		Top left two images: standard PCA. 
		Top right two images: PCA-L1.
		Bottom left two images: TRPCA.
		Bottom right two images: our approach.}\label{fig:face_rec}
	\end{figure}
	
\section{Conclusions} \label{sec:concl}
We proposed a Weiszfeld-like algorithm to address the robust PCA problem arising 
from a minimal distance function of lines from points
and gave a circumvent convergence analysis of the algorithm.
We will generalize this to multiple directions by considering the minimization on Stiefel manifolds,
resp. Grassmannians, where we have already recognized that several methods proposed in the literature just coincide
from the point of view of Grassmannians.
Further extensions of our findings are possible such as
the treatment of robust independent component analysis (ICA) and PCA on manifolds, see, e.g. \cite{FJ2004,FLPJ2004,HZ2006,P2016,P2017,SLN2014}.
Another modification of PCA, the so-called sparse PCA couples the data term \eqref{PCA_1} with a sparsity term for $A \in \mathbb R^{d,K}$, see, e.g., \cite{HPZ16,LT13}
and could be considered under the robustness point of view.

\subsection*{Acknowledgments} 
        
	Funding by the German Research Foundation (DFG)  with\-in the Research Training Group 1932,
	project area P3, is gratefully acknowledged.
	
\appendix	
\section{Appendix: One-Sided Derivatives and Minimizers on Embedded Manifolds} \label{sec:app}
The \emph{one-sided directional derivative} of a function $f: \mathbb R^d \rightarrow \mathbb R$, $d \in \mathbb N$,
at a point $x \in \mathbb R^d$ in direction $h \in \mathbb R^d$ is defined by
\[
	Df(x;h) := \lim_{\alpha\downarrow 0}\frac{f(x+\alpha h) - f(x)}{\alpha}.
\]
Restricting $f$ to a submanifold $\mathcal{M} \subseteq \mathbb R^d$, we can restrict our considerations
to $h \in T_x \mathcal{M}$.
Recall that $\mathcal{M} \subseteq \mathbb R^d$ is an $m$-dimensional submanifold of $\mathbb R^d$ 
if for each point $x \in \mathcal{M}$ there exists an open neighborhood $U \subseteq \mathbb R^d$ 
as well as an open set $\Omega \subseteq \mathbb R^m$ 
and a so-called parametrization $\varphi \in C^1(\Omega, \mathbb R^d)$ of $\mathcal{M}$ 
with the properties
\begin{itemize}
\item[i)]
$\varphi(\Omega) = \mathcal{M} \cap U$, 
\item[ii)] $\varphi^{-1}: \mathcal{M} \cap U \rightarrow \Omega$ 
is surjective and continuous, and 
\item[iii)]
$D \varphi(x)$ has full rank $m$ for all $x \in \Omega$.
\end{itemize}

To establish the relation between one-sided directional derivatives and local minima 
of functions on manifolds we need the following lemma. A proof can be found in \cite[Lemma B.1]{NNSS19}.

\begin{lemma}\label{tangent_cone}
	Let $\mathcal{M}\subset \RR^{d}$ be an $m$-dimensional manifold of $\RR^{d}$.
	Then the tangent space $T_{x}\mathcal{M}$ and the tangent cone
	\[
		\mathcal{T}_x\mathcal{M} 
		:= \left\{ u\in\RR^{d}:\ \exists \textup{ sequence } (x_k)_{k\in\NN}\subset\mathcal{M}\setminus \{x\} \textup{ with } x_k\rightarrow x \textup{ s.t. } 
		\frac{x_k-x}{\|x_k-x\|}\rightarrow\frac{u}{\|u\|}\right\} \cup \{0\}
	\]
	coincide.
\end{lemma}

The following theorem gives a general necessary and sufficient condition
for local minimizers of Lipschitz continuous functions on embedded manifolds 
using the notation of one-sided derivatives.
For the Euclidean setting $\mathcal{M} = \mathbb R^d$,
the first relation of the proposition is trivially fulfilled for any function  $f: \mathbb R^d \rightarrow \mathbb R$,
while a proof of the sufficient minimality condition in the second part  was given in \cite{Ben-Tal1985}.
Moreover, the authors of \cite{Ben-Tal1985} gave an example
that Lipschitz continuity in the second part cannot be weakened to just continuity.

\begin{theorem}\label{prop:loc_min}
        Let $\mathcal{M} \subset \mathbb R^d$ be an $m$-dimensional submanifold of $\mathbb R^d$
        and $f: \mathbb R^d \rightarrow \mathbb R$ a locally Lipschitz continuous function.
        Then the following holds true:
        \begin{enumerate}
	\item If  $\hat x \in \mathcal{M}$ is a local minimizer of  $f$ on $\mathcal{M}$, then
		$Df(\hat x;h) \ge 0$ for all $h \in T_{\hat x}\mathcal{M}$. 
	\item
	If
	$Df(\hat x;h) >0$ for all $h\in T_{\hat x} \mathcal{M} \setminus \{\mathbf{0}\}$,
	then $\hat x$ is a strict local minimizer of $f$ on $\mathcal{M}$.
	\end{enumerate}
\end{theorem}

A proof can be found in \cite[Thm. 6.1]{NNSS19} along with an example which demonstrates the necessity of the Lipschitz continuity of $f$ in the manifold setting in the first part of the theorem. Furthermore, note that $Df(\hat x;h) \ge 0$ for all $h\in T_{\hat x} \mathcal{M} \setminus \{\mathbf{0}\}$ does not imply that $\hat x$ is a local minimizer of $f$ on $\mathcal{M}$.

\bibliographystyle{abbrv}
\bibliography{refs_robpca}	
\end{document}